\documentclass[11pt,reqno]{amsart}
\usepackage[T1]{fontenc}
\usepackage[utf8]{inputenc}
\usepackage[a4paper,margin=3cm]{geometry}
\usepackage{mathtools, amssymb, amsthm, mathrsfs, xcolor, float}
\usepackage[shortlabels]{enumitem}
\usepackage{microtype}
\usepackage[colorlinks,
breaklinks,
linkcolor=red!80!black,
anchorcolor=green,
citecolor=blue!60!black, 
]{hyperref}

\definecolor{bleurr}{RGB}{36,82,135}
\def\bleurr{\color{bleurr}}

\usepackage{etoolbox}
\patchcmd{\section}{\normalfont}{\normalfont \bleurr}{}{}
\patchcmd{\subsection}{\normalfont}{\normalfont \bleurr}{}{}
\patchcmd{\subsubsection}{\normalfont}{\normalfont \bleurr}{}{}

\newtheorem{THEalpha}{Theorem}%[section]
\newtheorem{The}{Theorem}[section]
\newtheorem{Pro}[The]{Proposition}
\newtheorem{Lem}[The]{Lemma}

\theoremstyle{definition}
\newtheorem{defn}{Definition}[section]
\newtheorem{Rem}{Remark}[section]

\newcommand{\T}{\mathbb T}
\newcommand{\R}{\mathbb R}
\newcommand{\Z}{\mathbb{Z}}

\newcommand{\HH}{\mathbf{H}}
\newcommand{\vep}{\varepsilon}

\newcommand{\Lip}{\mathrm{Lip}}
\newcommand{\cM}{\mathcal{M}}
\newcommand{\cH}{\mathcal{H}}
\newcommand{\cL}{\mathcal{L}}

\newcommand{\sym}{\mathcal{S}}

\newcommand{\calE}{\mathcal{E}}

\newcommand{\tr}{\mathrm{tr}}
\newcommand{\trans}{{\mathrm{T}}}  % transpose
\newcommand{\diff}{\mathop{}\!\mathrm{d}}
\newcommand{\supp}{\operatorname{supp}}
\newcommand{\wee}{w^{\eta,\varepsilon}}
\newcommand{\wreg}{w_\eta^{\mathrm{reg}}}

\def\leq{\leqslant}
\def\le{\leqslant}
\def\geq{\geqslant}
\def\ge{\geqslant}
\def\tilde{\widetilde}
\def\hat{\widehat}
\title[]{A new selection problem for degenerate viscous Hamilton--Jacobi equations}

\author{Qinbo Chen}
\address[Qinbo Chen]{School of Mathematics, Nanjing University, Nanjing 210093, China}
\email{qinbochen@nju.edu.cn}

\author{Zhi-Xiang Zhu}
\address[Zhi-Xiang Zhu]{School of Mathematics, Nanjing University, Nanjing 210093, China}
\email{zhuzhixiang0928@gmail.com}


\begin{document}

\begin{abstract}	
We study a selection problem for degenerate viscous Hamilton--Jacobi equations with convex Hamiltonians, in which the approximation procedure combines a nonlinear discounted approximation with a small potential perturbation. A key question is how their simultaneous effects influence the asymptotic selection of viscosity solutions of the associated ergodic problem. Based on the nonlinear adjoint method, we establish the uniform convergence of the approximating solutions to a distinguished solution of the ergodic problem and derive a formula for the selected limit in terms of generalized Mather measures and the potential. As an application, we show that this selection principle is sufficiently flexible to realize any prescribed solution of the ergodic problem, with an explicit convergence rate.
\end{abstract}

\maketitle

\section{Introduction}
This paper studies a new selection problem for degenerate viscous Hamilton--Jacobi equations. More precisely, we consider the ergodic problem
\begin{equation}\label{eq_intro2}\tag{E}
    H(x, Du)=\tr\big(A(x) D^2 u\big)+c_H \quad \text{in}~\T^n
\end{equation}
where the diffusion matrix $A(x)$ is nonnegative definite, $H:\T^n\times\R^n\to\R$ is a convex Hamiltonian, $\T^n=\R^n/\Z^n$ is the flat torus, and $c_H\in\R$ is the ergodic constant. Owing to the possible degeneracy of $A$, viscosity solutions of \eqref{eq_intro2} are in general not unique, even modulo additive constants. This non-uniqueness thus raises a natural and fundamental question: \emph{whether the solutions produced by a given approximation procedure converge to a unique limit and, if so, which solution of the ergodic problem \eqref{eq_intro2} is selected}.

In this paper, we investigate the asymptotic behavior, as $\lambda\to 0^+$, of solutions $u_\lambda$ to the perturbed equation
\begin{equation}\label{eq_intro0}\tag{E$_\lambda$}
    f(x,\lambda u_\lambda)+ H(x, Du_\lambda)+\lambda V(x)=\tr\big( A(x) D^{2}u_\lambda\big)+c_H \quad \text{in}~\T^n.
\end{equation}
Here the perturbation consists of two parts: a nonlinear discounted approximation $f(x,\lambda u_\lambda)$ and a small potential perturbation $\lambda V(x)$. The main issue is to understand how their simultaneous effects determine the asymptotic limit of $\{u_\lambda\}$. Observe that this framework also includes the classical discounted approximation as a special case, namely when $f(x,\lambda u_\lambda)=\lambda u_\lambda$ and $V\equiv0$.

When $V(x)\equiv 0$, the above problem reduces to the vanishing discount problem. The interest in this problem originates from homogenization theory, stochastic optimal control, and differential games, see \cite{Lions_Papanicolaou_Varadhan1987, Bardi_Capuzzo-Dolcetta1997,  Arisawa_Lions1998} and the references therein. Lions, Papanicolaou, and Varadhan \cite{Lions_Papanicolaou_Varadhan1987} proved the relative compactness of discounted solutions, which gives convergence along subsequences. The uniqueness of the limit, however, remained a subtle issue. Substantial progress was first made in the first-order setting (i.e., when $A(x)\equiv 0$), where the problem is closely connected with dynamical systems. Gomes \cite{Gomes2008} and Iturriaga and S\'anchez-Morgado \cite{Iturriaga_Sanchez-Morgado2011} obtained partial results on possible limits. A major breakthrough was then achieved by Davini, Fathi, Iturriaga, and Zavidovique \cite{DFIZ2016}, who used weak KAM theory \cite{Fathi_book} to establish convergence to a unique limit and characterize it in terms of Mather measures. In this sense, the vanishing discount limit selects a particular critical solution of the ergodic problem. In the second-order setting, Mitake and Tran \cite{Mitake_Tran2017} established convergence for degenerate viscous Hamilton–Jacobi equations with scalar diffusion $A(x)=a(x) I_n$ by using the nonlinear adjoint method \cite{Evans_adjoint2010, Tran_book2021}. Ishii, Mitake, and Tran \cite{Ishii_Mitake_Tran2017_1,Ishii_Mitake_Tran2017_2} later developed a duality approach applicable to fully nonlinear second-order equations. 

Since then, the selection problem for discounted approximations and its variants have become an active topic, with many subsequent developments in diverse settings, including nonlinear discounted HJ equations \cite{Gomes_Mitake_Tran2018, CCIZ2019, WYZ_2021, Zavidovique2022_degenerate, QC2023, Zhang2024limit, Chen_Fathi_Za_Zha2024, DNYZ2024, QC2024}, weakly coupled systems \cite{Davini_Zavidovique2017, Ishii2019vanishing, Ishii_Jin2020}, discrete problems \cite{DFIZ_Math_Z_2016, Zavidovique_book2025, Ni_Zavidovique25}, mean field games \cite{Cardaliaguet_Porretta_MFG19, Iturriaga_Mendico_Wang_Xu25}, noncompact domains \cite{Ishii_Siconolfi2020, Davini2022, Tu_Zhang25vanishing}, negative discounting \cite{Davini_Wang2020, WYZ_negative}, state-constraint equations \cite{SonTu2022, Tu_Zhang2024}, non-monotone vanishing discount limits \cite{DNYZ2024, Ni_Yan_Zavidovique26}, nonlocal HJ equations \cite{Davini_Ishii25}, and vanishing discount-viscosity limits \cite{Wang_Zhang25}.

In contrast, equation \eqref{eq_intro0} in the present paper contains an additional source of selection, namely the small potential perturbation $\lambda V(x)$. Although this term vanishes in the limiting equation, it affects the asymptotic properties of the solutions and leads to a new selection problem.

The \textbf{motivation} for introducing $\lambda V(x)$ is twofold. From a physical and dynamical point of view, small potential perturbations are natural in the study of Hamiltonian systems and may affect their qualitative behavior. From the viewpoint of selection problems, the potential perturbation enlarges the class of accessible limit solutions beyond what can be achieved by discounted approximations alone. Indeed, the vanishing discount approximation selects a particular critical solution of the ergodic problem. Even if one allows more general discount structures (see e.g. \cite{WYZ_2021, QC2023, Chen_Fathi_Za_Zha2024}), the resulting selection mechanism is, in general, still not flexible enough to recover an arbitrary prescribed solution of the ergodic problem.

To illustrate this limitation concretely, consider $A(x)\equiv 0$ and a mechanical Hamiltonian on $\T=\R/\Z$ given by
\begin{equation*}
    H(x,p)=\frac{1}{2}p^2+\cos4\pi x.
\end{equation*}
The associated dynamics has two ergodic Mather measures, supported at the hyperbolic equilibria $(0,0)$ and $(\frac{1}{2},0)$. By \cite[Thm 1.5]{Chen_Fathi_Za_Zha2024}, the selected limit in this example is independent of the choice of the discount structure. This shows that varying the discount structure alone is not sufficient to select an arbitrary prescribed solution of the ergodic problem.

Therefore, introducing the potential perturbation is not merely a technical modification, but is a crucial ingredient in obtaining a more flexible selection principle. The first-order analogue of this new selection problem, corresponding to $A(x)\equiv0$, was studied by the first author in \cite{QC2024} using dynamical methods. The present paper addresses the second-order case, where the presence of a general, possibly degenerate diffusion matrix $A(x)$ requires different analytic tools.

\subsection{Main Results}
Throughout this paper, we impose the following assumptions.

\begin{itemize}
    \item[\textbf{(H1)}] The Hamiltonian $H\in C^{2}(\T^n\times\R^n)$ is strictly convex in $p$ and is superlinear in $p$. Moreover, there exists a positive constant $\gamma_{1}>0$ such that 
        \[
        |D_{x}H(x,p)|\leq \gamma_{1}\,\big(1+|H(x,p)|\big).
        \]
    \item[\textbf{(H2)}] The diffusion matrix $A(x)=\bigl(a_{ij}(x)\bigr)_{n\times n}\in C^2(\T^n)$ is positive semidefinite.
    \item[\textbf{(H3)}] The perturbation terms satisfy $V\in \Lip(\T^n)$ and $f(x,r)\in C^{1}(\T^n\times\R )$ with $\partial_{r}f>0$.  We also assume, without loss of generality, that  $f(x,0)\equiv 0$ on $\T^n$.
    \end{itemize}

\smallskip

Typical examples of $f$ satisfying (H3) include $f(x,r)=r$, corresponding to the standard discounted approximation, and $f(x,r)=\sigma(x)r$ with $\sigma\in C^1(\T^n)$ and $\sigma>0$, corresponding to a spatially varying discount factor. 
\begin{Rem}
In assumption (H3), the normalization $f(x,0)\equiv 0$ is imposed only for simplicity.  This is natural since, as $\lambda\to 0$, \eqref{eq_intro0} is intended as a perturbation of the ergodic problem \eqref{eq_intro2}. Formally, one may reduce to this case by replacing $f(x,r)$ with $f(x,r)-f(x,0)$ and $H(x,p)$ with $H(x,p)+f(x,0)$. We also remark that, without the monotonicity  condition $\partial_r f>0$, divergence phenomena may occur, see \cite{DNYZ2024}.

Under assumptions (H1)--(H2), the ergodic constant $c_H$ in \eqref{eq_intro2} is uniquely determined by $H$ and $A$. It is the only constant for which \eqref{eq_intro2} admits viscosity solutions. See e.g. \cite{Mitake_Tran2015, Cagnetti_Gomes_Mitake_Tran2015}.
\end{Rem}

Our first main result is the following.
\begin{THEalpha}\label{mainresult1}
Assume that {\rm (H1)--(H3)} hold.  Then there exists $\lambda_0>0$ such that, for every $\lambda\in(0,\lambda_0]$, the perturbed equation
\begin{equation}\label{eq_intro1}\tag{E$_\lambda$}
    f(x,\lambda u_\lambda)+ H(x, Du_\lambda)+\lambda V(x)=\tr \big( A(x) D^{2}u_\lambda\big)+c_H \quad \text{in}~ \T^n
\end{equation}
admits a unique Lipschitz continuous viscosity solution $u_\lambda: \T^n\to \R$. Moreover, as $\lambda\to 0^+$, $u_\lambda$ converges uniformly to a particular solution $u_0: \T^n\to \R$ of the ergodic problem \eqref{eq_intro2}, which is characterized by
\[u_0(x) = \sup_{w\in \calE} w(x), \quad x\in \T^n\] 
where $\calE$ denotes the collection of all continuous viscosity solutions $w$ of \eqref{eq_intro2} that satisfy the integral inequality
\begin{equation*}
\int_{\T^n\times\R^n} \partial_r f(x,0)\, w(x) \diff \tilde\mu \leq  -\int_{\T^n\times\R^n}V(x) \diff \tilde\mu\, \quad\text{for all}~\tilde\mu\in\tilde\cM_{\rm g},
\end{equation*}
and $\tilde\cM_{\rm g}$ is the collection of generalized Mather measures defined in \eqref{Mg_measures}. 
\end{THEalpha}

The characterization of $u_0$ as the supremum over a constrained class of solutions makes the role of the potential $V$ explicit in the selection mechanism.
\begin{Rem}
Generalized Mather measures were introduced in the study of the stochastic Mather problem \cite{Gomes2002_stochasticMather, Gomes2008} (see also \cite{Gomes_Mitake_Tran2022}). This notion can be viewed as a second-order analogue of the classical Mather measures in Aubry--Mather theory. In the present paper, the measures in $\tilde\cM_{\rm g}$ are constructed through an approximation procedure based on the nonlinear adjoint method. We refer to Subsection \ref{subsection_MatherMeasures} for the precise construction.
\end{Rem}

\begin{Rem}[Size of the potential perturbation]
We emphasize that the potential perturbation $\lambda V$ is chosen to have the same order as the leading term of the discounted approximation. This scale is critical for the present selection problem, as will be seen from the proofs in Section \ref{sec_proofofmainresult}. Indeed, if $\lambda V$ is replaced by $\lambda^\alpha V$, then for $\alpha>1$ the perturbation is too small to affect the selection of the limit, while for $\alpha<1$ it may destroy the asymptotic convergence of the solutions. Thus $O(\lambda)$ is the critical, and in this sense optimal, scale at which the potential perturbation can affect the selected limit.
\end{Rem}

We next discuss two consequences of Theorem \ref{mainresult1}. The first is a comparison principle for solutions of the ergodic problem, which provides a uniqueness criterion in terms of generalized Mather measures.

\begin{THEalpha}\label{mainresult2}
Assume that {\rm (H1)--(H2)} hold. Let $\sigma\in C(\T^n)$ be positive. For any two continuous viscosity solutions $u_1$ and $u_2$ of the ergodic problem \eqref{eq_intro2}, if 
\begin{equation}\label{constraint0}
	\int_{\T^n\times\R^n} u_1\, \sigma\, \diff \tilde\mu\leq \int_{\T^n\times\R^n} u_2\, \sigma\, \diff \tilde\mu \quad\text{for all}~ \tilde\mu\in \tilde\cM,
\end{equation}
where $\tilde\cM$ is the set of all generalized Mather measures, 
then $u_1\leq u_2$ on $\T^n$.	
\end{THEalpha}

In particular, by taking $\sigma\equiv1$ in Theorem \ref{mainresult2}, we see that if two continuous viscosity solutions $u_1$ and $u_2$ of the ergodic problem \eqref{eq_intro2} satisfy 
	\begin{equation*}
	\int u_1 \diff \tilde\mu\leq \int u_2 \diff\tilde\mu \quad\text{for all}~ \tilde\mu\in \tilde\cM,
\end{equation*}
then $u_1\leq u_2$ on $\T^n$. Hence, a viscosity solution of equation \eqref{eq_intro2} is uniquely determined by its integrals against all generalized Mather measures. This recovers the uniqueness criterion established by Mitake and Tran \cite[Section 4]{Mitake_Tran_uniqueness_set}, with slightly weaker assumptions on $H$.

\smallskip

While Theorem \ref{mainresult1} describes the selected limit for a fixed potential $V$, the next result shows that the potential can serve as a control parameter: by varying $V$, the corresponding solutions $u_\lambda$ can be made to converge to any prescribed solution of the ergodic problem \eqref{eq_intro2}.

\begin{THEalpha}\label{mainresult3}
Assume that {\rm (H1)--(H2)} hold, and let $f\in C^{1,1}_{\rm loc}(\T^n\times\R)$ be a given function satisfying the condition imposed on $f$ in {\rm (H3)}. 

Then, for each viscosity solution $\hat{u}_0\in C(\T^n)$ of the ergodic problem \eqref{eq_intro2}, there exists a potential $\hat{V}_0\in \Lip(\T^n)$ such that, for all sufficiently small $\lambda>0$, the perturbed equation
\begin{equation}
    f(x,\lambda u_\lambda)+ H(x, Du_\lambda)+\lambda \hat{V}_0(x)=\tr\big( A(x) D^{2}u_\lambda\big)+c_H \quad \text{in}~\T^n,
\end{equation}
admits a unique viscosity solution $u_\lambda\in \Lip(\T^n)$, which converges uniformly to the prescribed solution $\hat{u}_0$ as $\lambda\to 0^+$, with convergence rate $O(\lambda)$, namely,  
\begin{equation*}
		\|u_\lambda-\hat{u}_0\|_{L^\infty(\T^n)}\,\leq \, C\lambda\,
\end{equation*}
for a constant $C$ independent of $\lambda$.
\end{THEalpha}

\begin{Rem}
We stress that there may be many potentials for which the corresponding solutions $u_\lambda$ converge to the prescribed solution $\widehat u_0$ as $\lambda\to 0$. Our particular choice of potential $\widehat V_0$ in Theorem \ref{mainresult3} is made so as to yield the sharp convergence rate of order $O(\lambda)$.  In contrast, obtaining convergence rates for general cases remains a rather open problem. Some special examples were studied in \cite{Mitake_Soga2018}.
\end{Rem}

In conclusion, the preceding results establish a new selection principle for the ergodic problem \eqref{eq_intro2}, which is flexible enough to realize every viscosity solution as a selected limit.

\subsection{Outline of the Proof}
We briefly describe the proof strategy and the organization of the paper.

The first step is to establish \emph{a priori} estimates for solutions of equation \eqref{eq_intro0}. We first derive a Lipschitz estimate for a class of degenerate second-order Hamilton--Jacobi equations by the doubling variables method, see Theorem \ref{The_Bernst_Lip}. This applies to \eqref{eq_intro0} and yields uniform Lipschitz bounds for $u_\lambda$. Together with Perron's method and the comparison principle, it gives existence, uniqueness, and uniform estimates for the family $\{u_\lambda\}$, see Theorem \ref{The_equiestimates}.

The second step is to construct generalized Mather measures through an approximation scheme based on the nonlinear adjoint method. These measures play a central role in the selection problem. The nonlinear adjoint method was originally introduced by Evans \cite{Evans_adjoint2010} (see also \cite{Tran_adjoint2011}). It was applied in \cite{Mitake_Tran2017} to study the vanishing discount limit in the scalar diffusion case $A(x)=a(x)I_n$. See \cite{Zhang2024limit} for nonlinear discounted equations with the same scalar diffusion. In Section \ref{section_generalized_Mathermeasures}, we adapt this method to handle the general diffusion matrix $A(x)$ and the additional potential perturbation appearing in the present selection problem.

The third step is the asymptotic analysis of the family $\{u_\lambda\}$ using the generalized Mather measures constructed above. The crucial point is to establish two key estimates, Lemmas \ref{new_gen_1} and \ref{new_gen_2}, formulated in terms of these measures, which encode the limiting constraints that determine the selected solution. One of the main technical difficulties is the lack of smoothness of viscosity subsolutions. More precisely, it is unclear whether every Lipschitz viscosity subsolution of the ergodic problem \eqref{eq_intro2} can be approximated by smooth subsolutions. The commutation lemma of \cite{Mitake_Tran2017} is one technical device to handle this issue in the scalar diffusion case, but it is not known whether an analogous result holds for a general diffusion matrix. We therefore use a different approximation procedure: we establish a weaker smooth approximation result (Theorem \ref{The_approx_sub}) adapted to the present setting. It is based on the Lasry--Lions regularization combined with an additional mollification, in the spirit of \cite{Gomes_Mitake_Tran2022}. We refer to Section \ref{sec_smoothapproximation} for further details.

With these preparations, Theorem \ref{mainresult1} is proved in Section \ref{sec_proofofmainresult}. Theorem \ref{mainresult2} and Theorem \ref{mainresult3} are then derived as applications of Theorem  \ref{mainresult1}.

%%%%%%%%%%%%%%%%%%%%%%%%%%%%%%%%%%%%%%%%%%%%%%%%%%%%%%%%%%%%%%%%%%
\section{A priori Estimates and Well-posedness}\label{sec_uniformestimates}
In this section, we establish the basic a priori estimates and well-posedness results needed in the sequel. We begin with a Lipschitz estimate for a class of degenerate second-order Hamilton–Jacobi equations. We then apply this estimate to equation \eqref{eq_intro1} to obtain existence, uniqueness, and uniform estimates for $u_\lambda$.

\subsection{Lipschitz Estimates for Degenerate Second-Order Equations}
We give gradient bounds for viscosity solutions of the following class of second-order, possibly degenerate, Hamilton–Jacobi equations:
\begin{equation}\label{eq_chenz1}
    -\tr \bigl(A(x)D^2u\bigr)+f(x,\delta u)+H(x,Du)+\delta V(x)=0\qquad\text{in}~\T^n 
\end{equation}
where $\delta\geq 0$. 

Although this result will be applied later to equation \eqref{eq_intro1}, we state it in a slightly more general form. Throughout this subsection, we only need the assumptions listed below:
\begin{itemize}
    \item  $H\in C^{1}(\T^{n}\times\R^{n})$ is superlinear in $p$, and there is a constant $\gamma_1>0$ such that
          \begin{equation}\label{condi_D_xH}
          	|D_{x}H(x,p)|\leq \gamma_1\big(1+|H(x,p)|\big).
          \end{equation}
    \item  $f(x,r)\in C^{1}(\T^{n}\times\R)$ is strictly increasing in $r$. The potential $V\in \Lip(\T^n)$.
    \item $A(x)=\bigl(a_{ij}(x)\bigr)_{n\times n}$ is positive semidefinite for every $x$, and $A\in C^{2}(\T^n)$.
\end{itemize}

We emphasize that no convexity assumption on $H(x,p)$ is needed here. Since $A$ is of class $C^{2}$, it admits a Lipschitz square root \cite[Theorem 5.2.3]{Stroock_Varadhan_book}: there exists a Lipschitz map $\Sigma: \T^{n}\to\sym^{n}$, where $\sym^n$ denotes the space of $n\times n$ real symmetric matrices, such that $\Sigma(x)$ is positive semidefinite and
\begin{equation}\label{squareofA}
	A(x)=\Sigma(x)^{\trans}\,\Sigma(x)\quad \text{for every}~ x\in \T^n.
\end{equation}

\begin{The}\label{The_Bernst_Lip}
Suppose that the assumptions listed above hold.
Let $\delta\ge 0$ and let $u\in C(\T^n)$ be a viscosity solution of \eqref{eq_chenz1}. Then $u$ is Lipschitz continuous, and its gradient satisfies
\[
|Du(x)|\le K \quad\text{for almost every}~x\in\T^n.
\]
Here, the Lipschitz constant $K$ is given by
\[
K := e^{2\gamma_1}\bigl(\gamma_2+2\|f\|_{C^1(U)}\bigr)
     +\delta\bigl[e^{2\gamma_1}\|V\|_{L^\infty}+\Lip (V)\bigr]+C,
\]
where the constant $\gamma_2:=1+2\left|\min_{(x,p)\in \T^n\times\R^n} H(x,p)\right|$, the set $U:=\big\{(x,r)\in \T^n\times \mathbb R: |r|\le\delta\|u\|_{L^\infty}\big\}$, and 
the constant $C=C(\gamma_1,\Lambda_0)$ depends only on $\gamma_1$ and $\Lambda_0:=\bigl(\Lip(\Sigma)\bigr)^2+\sup_{x\in \T^n}\|\Sigma(x)\|^2_F$ with $\|\cdot\|_F$ being the Frobenius norm of a matrix. 
\end{The}

Our proof mainly follows the strategy of \cite{Holderestimates_for_degenerate_elliptic, Armstrong_Tran2015}. We include the argument here for completeness, since in our setting the discount term  $f(x, \delta u)$ is nonlinear and the structural condition \eqref{condi_D_xH} on $H$ differs slightly from those in \cite{Holderestimates_for_degenerate_elliptic, Armstrong_Tran2015}.  In particular, under assumption \eqref{condi_D_xH}, Gronwall's inequality is needed in the proof of Lemma \ref{Lem_t2HfV}.

Before giving the proof, we fix the following notation. For symmetric matrices $A,B\in\sym^n$, we write $A\succ B$ (resp. $A\succeq B$) if $A-B$ is positive definite (resp. positive semidefinite). This is commonly referred to as the L\"owner partial order. 
\begin{proof}[Proof of Theorem \ref{The_Bernst_Lip}]
Our aim is to show that there exists a large constant $K>0$ so that
\[
u(x)-u(y)-K|x-y|\leq 0\qquad\text{for all }x,y\in\R^n.
\]
Here, $u:\T^n\to\R$ is regarded as a $\Z^n$-periodic function on $\R^n$. We argue by contradiction. Suppose that for some $L>0$
\begin{equation}\label{sup_forL}
\sup_{x,y\in\R^n}\bigl(u(x)-u(y)-L|x-y|\bigr) > 0.
\end{equation}
We will then show that this forces $L$ to satisfy an upper bound, and hence $L$ cannot be arbitrarily large.

By the $\Z^n$-periodicity of $u$, the sup can be attained at points $x_0, y_0\in \overline{B}_1$\, (the unit ball),
\begin{equation}\label{contr_assup}
u(x_0)-u(y_0)-\Phi_L(x_0,y_0)=\sup_{x,y\in\R^n}\bigl(u(x)-u(y)-\Phi_L(x,y)\bigr)>0,
\end{equation}
where
\[\Phi_L(x,y):=L|x-y|.\] Inequality \eqref{contr_assup} forces $x_0\neq y_0$. Hence $\Phi_L$ is $C^2$ in a neighborhood of $(x_0,y_0)$. By the Crandall-Ishii Lemma, see \cite[Theorem 3.2]{Userguide92}, there exist matrices $X_\varepsilon,Y_\varepsilon\in\sym^n$  for every $\varepsilon>0$ such that
\[
(D_x\Phi_L(x_0,y_0),X_\varepsilon)\in\overline{J}^{2,+}u(x_0)\,,\qquad
(-D_y\Phi_L(x_0,y_0),Y_\varepsilon)\in\overline{J}^{2,-}u(y_0)\,,
\]
for which the following matrix inequality holds
\begin{equation}\label{matr_upb}
\begin{pmatrix}
X_\varepsilon & 0 \\
0 & -Y_\varepsilon
\end{pmatrix}\preceq B+\varepsilon B^2,
\quad\text{where }B:=D^2\Phi_L(x_0,y_0)\in\sym^{2n}.
\end{equation}
Here, $\overline{J}^{2,+}u(x)$ (resp. $\overline{J}^{2,-}u(x)$) denotes the closure of the second‑order superjet (resp. subjet) of $u$ at $x$, see \cite{Userguide92}. Moreover, since $x_0\neq y_0$, we have the explicit formulas
\begin{equation}\label{djkq0}
D_x\Phi_L(x_0,y_0)=-D_y\Phi_L(x_0,y_0)=L\mathbf{e}_0\,,\qquad
B=\frac{L}{|x_0-y_0|}\begin{pmatrix}
Z & -Z \\ -Z & Z
\end{pmatrix},
\end{equation}
where $\displaystyle{\mathbf{e}_0=\frac{x_0-y_0}{|x_0-y_0|}}$ is a column vector with $|\mathbf{e}_0|=1$, and
\begin{equation}\label{formu_Z}
	Z=I_n-\mathbf{e}_0\otimes\mathbf{e}_0=I_n-\mathbf{e}_0\mathbf{e}_0^{\trans}.
\end{equation}

Since $u$ is a viscosity solution of \eqref{eq_chenz1}, we have
\begin{equation}\label{sa_w0}
\begin{aligned}
-\tr\bigl(A(x_0)X_\varepsilon\bigr)+f(x_0,\delta u(x_0))+H(x_0,L\mathbf{e}_0)+\delta V(x_0) &\le 0,\\[2pt]
-\tr\bigl(A(y_0)Y_\varepsilon\bigr)+f(y_0,\delta u(y_0))+H(y_0,L\mathbf{e}_0)+\delta V(y_0) &\ge 0.
\end{aligned}
\end{equation}
Recall that $A(x)=\Sigma(x)^{\trans}\Sigma(x)$ by \eqref{squareofA}. For any $t>1$, we define 
\[
G_t:=\begin{pmatrix}t\Sigma(x_0) & \Sigma(y_0)\end{pmatrix}^\trans
      \begin{pmatrix}t\Sigma(x_0) & \Sigma(y_0)\end{pmatrix}
    =\begin{pmatrix}
t^2\Sigma(x_0)^\trans\Sigma(x_0) & t\Sigma(x_0)^\trans\Sigma(y_0)\\
t\Sigma(y_0)^\trans\Sigma(x_0) & \Sigma(y_0)^\trans\Sigma(y_0)
\end{pmatrix}\,
\]
which is a positive semidefinite matrix. Multiplying \eqref{matr_upb} by $G_t$ and taking the trace yields
\[
t^2\tr\bigl(A(x_0)X_\varepsilon\bigr)-\tr\bigl(A(y_0)Y_\varepsilon\bigr)
\le \tr(G_tB)+\varepsilon\tr(G_tB^2).
\]
Combining this with \eqref{sa_w0}, we arrive at
\begin{align*}
\tr(G_tB)+\varepsilon\tr(G_tB^2)
&\ge ~t^2H(x_0,L\mathbf{e}_0)-H(y_0,L\mathbf{e}_0)\\
&~\quad +t^2f(x_0,\delta u(x_0))-f(y_0,\delta u(y_0))
   +t^2\delta V(x_0)-\delta V(y_0).
\end{align*}
Letting $\varepsilon\to0^+$, we get the key inequality 
\begin{equation}\label{trGtB0}
\begin{aligned}
\tr(G_tB)& \ge ~ t^2H(x_0,L\mathbf{e}_0)-H(y_0,L\mathbf{e}_0)\\[2pt]
         &~\quad +t^2f(x_0,\delta u(x_0))-f(y_0,\delta u(y_0))+t^2\delta V(x_0)-\delta V(y_0).
\end{aligned}
\end{equation}

\,

We now estimate the two sides of \eqref{trGtB0} separately. The left-hand side is controlled by the following lemma.

\begin{Lem}\label{Lem_GtB_uppbd}
For any $t>1$,
\begin{equation}\label{trGtB_uppbd}
    \tr\bigl(G_t B\bigr) \le 2L\Lambda_0\left(\frac{(t^2-1)^2}{|x_0-y_0|}+|x_0-y_0|\right),
\end{equation}
where the constant $\Lambda_0:=\bigl(\Lip(\Sigma)\bigr)^2+\sup_{x\in \T^n}\|\Sigma(x)\|^2_F$.
\end{Lem}

\begin{proof}
Recall from \eqref{formu_Z} that $Z=I_n-\mathbf{e}_0\otimes\mathbf{e}_0$, so $Z$ is an orthogonal projection, satisfying $Z=Z^2$ and $0\preceq Z\preceq I_n$.  Using the explicit form \eqref{djkq0} of $B$, we compute 
\begin{align*}
\tr(G_t B)
&=\frac{L}{|x_0-y_0|}\,\tr\left[\bigl(t\Sigma(x_0)-\Sigma(y_0)\bigr)^\trans
   \bigl(t\Sigma(x_0)-\Sigma(y_0)\bigr) \,Z\right] \\
&=\frac{L}{|x_0-y_0|}\,\tr\left[\bigl(t\Sigma(x_0)-\Sigma(y_0)\bigr) \,Z\,
   \bigl(t\Sigma(x_0)-\Sigma(y_0)\bigr)^\trans \right] \\
&\leq \frac{L}{|x_0-y_0|}\,\tr\left[\bigl(t\Sigma(x_0)-\Sigma(y_0)\bigr)\,
   \bigl(t\Sigma(x_0)-\Sigma(y_0)\bigr)^\trans \right] \\
&=\frac{L}{|x_0-y_0|}\,\left\|t\Sigma(x_0)-\Sigma(y_0)\right\|_F^2
\end{align*}
where $\|\cdot\|_F$ is the Frobenius norm for a matrix. 

Applying the inequality $\|a+b\|_F^2\le2(\|a\|_F^2+\|b\|_F^2)$ with $a=(t-1)\Sigma(x_0)$ and $b=\Sigma(x_0)-\Sigma(y_0)$, we obtain
\begin{align*}
\tr(G_t B)\leq \frac{2L}{|x_0-y_0|} \Bigl((t-1)^2\|\Sigma(x_0)\|_F^2+\|\Sigma(x_0)-\Sigma(y_0)\|_F^2\Bigr).
\end{align*}
Since $\Sigma$ is Lipschitz and $(t-1)^2\le(t^2-1)^2$ for $t>1$, the desired estimate \eqref{trGtB_uppbd} follows. 
\end{proof}

We next turn to the right-hand side of the key inequality \eqref{trGtB0}.

\begin{Lem}\label{Lem_t2HfV}
For every $t>1$, the following estimates hold:
\begin{align}
t^2H(x_0, L\mathbf{e}_0)-H(y_0, L\mathbf{e}_0)
&\ge ~\bigl(t^2-e^{\gamma_1|x_0-y_0|}\bigr)H(x_0, L\mathbf{e}_0)
    -\gamma_2\bigl(e^{\gamma_1|x_0-y_0|}-1\bigr)\,,\label{t2H_lowbd}
\end{align}
where the constant $\gamma_2:=1+2\left|\min_{(x,p)\in \T^n\times\R^n} H(x,p)\right|$.
\begin{align}
t^2f(x_0,\delta u(x_0))-f(y_0, \delta u(y_0))
&\ge ~-(t^2-1)\bigl\|f|_U\bigr\|_{L^\infty}
    -\bigl\|D_x f|_U\bigr\|_{L^\infty}\,|x_0-y_0|\,,\label{t2f_lowbd}\\[2mm] 
t^2\delta V(x_0)-\delta V(y_0)
&\ge ~ -(t^2-1)\delta\|V\|_{L^\infty}
    -\delta\,\Lip(V)\,|x_0-y_0|\,,\label{t2V_lowbd}
\end{align}
where $f|_U$ and $D_x f|_U$ denote the restrictions of $f$ and $D_x f$ to the set  
$U=\big\{(x,r)\in \T^n\times \R : |r|\le\delta\|u\|_{L^\infty}\big\}$, respectively.
\end{Lem}

\begin{proof}
Define $h(\tau)=H\big(x_0+\tau(y_0-x_0), L\mathbf{e}_0\big)$ for $\tau\in[0,1]$.  Using condition \eqref{condi_D_xH}, namely, $|D_{x}H(x,p)|\leq \gamma_1\big(1+|H(x,p)|\big)$, we obtain the differential inequality  
\[h'(\tau)\le \gamma_1\Bigl(1+|h (\tau)|\Bigr)|x_0-y_0|.\] 
By the superlinearity of $H$, $H$ is bounded from below, i.e., $\min H(x,p)>-\infty$. This implies $|h(\tau)|\leq h(\tau)+2\left|\min H(x,p)\right|$, and therefore,
\begin{equation}\label{diffineq_h}
h'(\tau)\le \gamma_1|x_0-y_0| \Bigl(\gamma_2+h (\tau)\Bigr),
\end{equation} 
where the constant $\gamma_2=1+2\left|\min H(x,p)\right|$. Applying Gronwall’s inequality to \eqref{diffineq_h} yields
\[
h(1)\le e^{\gamma_1|x_0-y_0|}h (0)+\gamma_2\bigl(e^{\gamma_1|x_0-y_0|}-1\bigr),
\] 
namely
\[
H(y_0, L\mathbf{e}_0)\le e^{\gamma_1|x_0-y_0|}H(x_0, L\mathbf{e}_0)
   +\gamma_2\bigl(e^{\gamma_1|x_0-y_0|}-1\bigr),
\]
which easily leads to \eqref{t2H_lowbd}.

To prove inequality \eqref{t2f_lowbd}, we write
\begin{align*}
t^2f(x_0,\delta u(x_0))-f(y_0,\delta u(y_0))  
&= (t^2-1)f(x_0,\delta u(x_0))
   +\bigl[f(x_0,\delta u(x_0))-f(y_0,\delta u(x_0))\bigr] \\[2pt]
&\qquad +\bigl[f(y_0,\delta u(x_0))-f(y_0,\delta u(y_0))\bigr].
\end{align*}
Note that $f(y_0,\delta u(x_0))-f(y_0,\delta u(y_0))$ is non‑negative, since $f$ is increasing in its second variable and $u(x_0)>u(y_0)$ by \eqref{contr_assup}. Hence,
\begin{align*}
t^2f(x_0,\delta u(x_0))-f(y_0,\delta u(y_0)) \ge -(t^2-1)\bigl\|f|_U\bigr\|_{L^\infty}
      -\bigl\|D_x f|_U\bigr\|_{L^\infty}\,|x_0-y_0|.
\end{align*}
where the set $U:=\big\{(x,r)\in \T^n\times \R : |r|\le\delta\|u\|_{L^\infty}\big\}$.

The last assertion \eqref{t2V_lowbd} follows directly from the Lipschitz continuity of $V$,
\begin{align*}
t^2\delta V(x_0)-\delta V(y_0)
&= (t^2-1)\delta V(x_0)+\delta\bigl(V(x_0)-V(y_0)\bigr) \\
&\ge -(t^2-1)\delta\|V\|_{L^\infty}
    -\delta\,\Lip(V)\,|x_0-y_0|. 
\end{align*}
\end{proof}

We are now ready to complete the proof of Theorem \ref{The_Bernst_Lip} by combining the previous estimates. Set
\[
t:=\sqrt{1+\sigma|x_0-y_0|},
\]
where $\sigma>0$ will be chosen later. 

Applying Lemmas \ref{Lem_GtB_uppbd} and \ref{Lem_t2HfV} to the key inequality \eqref{trGtB0} and dividing by $|x_0-y_0|$,
\begin{equation*}
2L\Lambda_0(1+\sigma^2)
\ge (\sigma-\theta_0)H(x_0,L\mathbf{e}_0)-\gamma_2\theta_0 -(\sigma+1)\|f\|_{C^1(U)}
   -\delta\bigl(\sigma\|V\|_{L^\infty}+\Lip(V)\bigr),
\end{equation*}
where $\|f\|_{C^1(U)}:=\|f|_U\|_{L^\infty}+\|D_x f|_U\|_{L^\infty}$ and the constant $\theta_0$ is 
\[
\theta_0:=\frac{e^{\gamma_1|x_0-y_0|}-1}{|x_0-y_0|}>0.
\]
Choose $\sigma:=\theta_0+1$. The above inequality then becomes
\begin{equation*}
4L \Lambda_0(1+\theta_0)^2
\ge H(x_0,L\mathbf{e}_0)-\gamma_2\theta_0 -(\theta_0+2)\|f\|_{C^1(U)}
   -\delta\bigl[ (\theta_0+1)\|V\|_{L^\infty}+\Lip(V)\bigr].
\end{equation*}
Note that $\theta_0$ satisfies
\[
\theta_0\le\frac{e^{2\gamma_1}-1}{2},
\]
since the function $s\mapsto(e^{\gamma_1 s}-1)/s$ is increasing on $[0,\infty)$ and $x_0,y_0\in \overline{B}_1$. Hence,
\begin{equation}\label{dkasd1}
e^{2\gamma_1}\bigl(\gamma_2+2\|f\|_{C^1(U)}\bigr)
+\delta\bigl[e^{2\gamma_1}\|V\|_{L^\infty}+\Lip(V)\bigr]
\ge H(x_0,L\mathbf{e}_0)-4L\Lambda_0 e^{4\gamma_1}.
\end{equation}

By the superlinearity of $H$, there exists a constant $C=C(\gamma_1,\Lambda_0)>0$ such that
\begin{equation}\label{H_suplin}
H(x,p)\ge\bigl(4\Lambda_0 e^{4\gamma_1}+1\bigr)|p|-C\qquad\text{for all }(x,p)\in\mathbb{T}^n\times\mathbb{R}^n.
\end{equation}
Combining the last two inequalities \eqref{dkasd1}-\eqref{H_suplin}, we conclude that 
\[
L\le e^{2\gamma_1}\bigl(\gamma_2+2\|f\|_{C^1(U)}\bigr)
      +\delta\bigl[e^{2\gamma_1}\|V\|_{L^\infty}+\Lip(V)\bigr]+C=:K.
\]
Thus the constant $L$ in \eqref{sup_forL} must satisfy $L\le K$. Consequently,
\[
u(x)-u(y)-K|x-y|\le 0\qquad\text{for all }x,y\in\mathbb{T}^n.
\]
This finally ends the proof of Theorem \ref{The_Bernst_Lip}.
\end{proof}

\,

\subsection{Well-posedness and Uniform Estimates for \texorpdfstring{\eqref{eq_intro1}}{}}
We use the Lipschitz estimate established above to prove existence and uniqueness of viscosity solutions to equation \eqref{eq_intro1}, together with uniform estimates that will be needed later in the asymptotic analysis.
 
\begin{The}\label{The_equiestimates}
Assume that {\rm (H1)--(H3)} hold. Then there exists $\lambda_0>0$ such that, for every $\lambda\in (0,\lambda_0]$, equation \eqref{eq_intro1}, namely
\begin{equation*}
    f(x,\lambda u_\lambda)+ H(x, Du_\lambda)+\lambda V(x)=\tr \bigl( A(x) D^2 u_\lambda\bigr)+c_H \qquad \text{in}~ \T^n
\end{equation*}
admits a unique continuous viscosity solution $u_\lambda: \T^n \to \R$.

Moreover, the family $\{u_\lambda\}_{\lambda \in (0,\lambda_0]}$ is equibounded and equi-Lipschitz on $\T^n$.
\end{The}

\begin{proof}
By Theorem \ref{The_Bernst_Lip}, any continuous viscosity solution of \eqref{eq_intro1} is Lipschitz continuous. This, together with the comparison principle, ensures the uniqueness of the solution. 

To establish existence we use Perron's method \cite{Userguide92}. We first note that the ergodic problem 
\begin{equation}\label{eq_ergsecoHJ}
    H(x, Du)=\tr \bigl(A(x) D^2 u\bigr)+c_H \qquad \text{in}~ \T^n.
\end{equation}
admits Lipschitz viscosity solutions, so we fix one such solution $\hat{u}:\T^n\to\R$ of \eqref{eq_ergsecoHJ}. We then define two functions $\hat{u}_-$ and $\hat{u}_+$ on $\T^n$ by
\begin{equation*}
\hat{u}_-(x):=\hat{u}(x)-\max\hat{u}-\frac{1+\|V\|_\infty}{d_0},\qquad
\hat{u}_+(x):=\hat{u}(x)-\min\hat{u}+\frac{1+\|V\|_\infty}{d_0}.
\end{equation*}
where the constant $d_0:=\min_{x\in\T^n} \partial_r f(x,0)>0$.

 We claim that $\hat{u}_-$ and $\hat{u}_+$ are, respectively, a viscosity subsolution and a supersolution of equation \eqref{eq_intro1} for all sufficiently small $\lambda>0$. It suffices to verify the subsolution property for $\hat{u}_-$, since the argument for $\hat{u}_+$ is analogous.   Notice that $\hat{u}_-$ differs from $\hat{u}$ only by a constant, so their derivatives coincide. Therefore, it remains to check that, for all sufficiently small $\lambda>0$,
\begin{equation}\label{checjk1}
	f(x,\lambda \hat{u}_-(x))+\lambda V(x)\leq 0\qquad \text{in}~\T^n.
\end{equation}
Recalling that $\partial_rf>0$ and $f(x,0)\equiv 0$ by assumption (H3), Taylor's formula yields
\begin{align*}
	 f(x,\lambda \hat{u}_-(x))+\lambda V(x)
	\leq & ~f\left(x, -\lambda\frac{1+\|V\|_\infty}{d_0} \right)+\lambda V(x)\\
	= & ~-\lambda \,\partial_r f(x,0)\,\frac{1+\|V\|_\infty}{d_0}+\lambda \,\omega(\lambda)+\lambda V(x)\\
	\leq &~-\lambda+\lambda\,\omega(\lambda)\,,
\end{align*}
for some modulus $\omega: [0,\infty)\to [0,\infty)$. Therefore, there exists $\lambda_0>0$ such that \eqref{checjk1} holds for all $\lambda\in(0,\lambda_0]$. This proves the claim.

Now, for each $\lambda\in(0,\lambda_0]$, invoking Perron's method we define
\begin{equation*}
u_\lambda(x) \overset{\Delta}{=} \sup\{ w(x) \mid \hat{u}_- \leq w \leq \hat{u}_+,\,  w~\text{is a continuous viscosity subsolution of \eqref{eq_intro1}} \}\,
\end{equation*}
Then $u_\lambda:\T^n\to\R$ is a continuous viscosity solution of equation \eqref{eq_intro1}. 

Finally, the equiboundedness of the family $\{u_\lambda\}_{\lambda \in (0,\lambda_0]}$ follows directly from the construction above,  while the equi-Lipschitz continuity follows from Theorem \ref{The_Bernst_Lip} together with the equiboundedness. 
\end{proof}
\begin{Rem}
We stress that Theorem \ref{The_equiestimates} may not hold for large $\lambda>\lambda_0$. See \cite[Section 3]{QC2024} for a counterexample.
\end{Rem}

%%%%%%%%%%%%%%%%%%%%%%%%%%%%%%%%%%%%%%%%%%%%%%%%%%%%%%%%%%%%%%%%%%
\section{Regularization Process and Construction of Generalized Mather Measures}\label{section_generalized_Mathermeasures}
In this section, we first regularize equation \eqref{eq_intro1} by the vanishing viscosity method in order to obtain smooth approximate solutions. We then introduce the adjoint equation associated with the regularized problem, and use the nonlinear adjoint method to construct a special class of probability measures that enter the selection formula in Theorem \ref{mainresult1}.

\subsection{Elliptic Regularization}
Solutions $u_\lambda$ of \eqref{eq_intro1} are generally only Lipschitz continuous, while the nonlinear adjoint method requires sufficient regularity. To overcome this, we  therefore introduce an elliptic regularization by adding a small viscosity term. This yields a family of uniformly elliptic equations whose solutions enjoy higher regularity and approximate $u_\lambda$ as the viscosity parameter tends to zero. This procedure is also known as the vanishing viscosity method.

More precisely, for each $\eta>0$, we consider the following elliptic regularization of \eqref{eq_intro1}:
\begin{equation}\label{eq_regularized}\tag{$\rm E^\eta_\lambda$}
    f(x,\lambda u^\eta_\lambda) + H(x, Du^\eta_\lambda) + \lambda V(x) 
    = \eta^{2}\Delta u_\lambda^{\eta} + \tr \bigl(A(x) D^{2}u^\eta_\lambda\bigr) + c_H \qquad \text{in}~ \T ^n.
\end{equation}
Then we have the following lemma.
\begin{Lem}\label{Lem_exst_ulameta}
There exist $\lambda_0>0$ and $\eta_0>0$ such that, for every  $\lambda\in (0,\lambda_0]$ and $\eta\in (0,\eta_0]$, equation \eqref{eq_regularized} has a unique $C^2$ solution $u^\eta_\lambda:\T ^n\to\R$. 

Moreover, $\lambda\|u_\lambda^\eta\|_{L^\infty(\T^n)}$ and $\|Du_\lambda^\eta\|_{L^\infty(\T^n)}$ are bounded uniformly in $\lambda\in (0,\lambda_0]$ and $\eta\in (0,\eta_0]$.
\end{Lem}
\begin{proof}
For every $\eta>0$, there exists a unique constant $c_H^\eta$ such that the ergodic problem
\begin{equation}\label{stateq1}	   
    -\eta^{2}\Delta v^\eta -\tr \bigl(A(x) D^2 v^\eta\bigr)+H(x, Dv^\eta) = c_H^\eta \qquad \text{in}~ \T ^n
\end{equation}
admits a solution $v^\eta\in C^2(\T^n)$, unique up to additive constants. By adding a constant if necessary, we may assume $\min_{x\in \T^n} v^\eta(x) = 0$. Moreover, it was shown in \cite{Gomes2002_stochasticMather, Cagnetti_Gomes_Mitake_Tran2015} that 
\begin{equation}\label{etaconv}
	\lim_{\eta\to 0^+}c_H^\eta=c_H,
\end{equation}
and $\{v^\eta\}_{\eta\in (0,1]}$ is equi-Lipschitz continuous. Hence, in view of $\min_{x\in \T^n} v^\eta(x) = 0$ and the compactness of $\T^n$, there exists  $K_0>0$ such that $\|v^\eta\|_\infty+\|Dv^\eta\|_\infty\leq K_0$ for all $\eta\in (0,1]$. 

We claim that the functions
\[
\overline u^\eta_\lambda:=v^\eta+\frac{K_0+1}{\lambda}
\qquad\text{and}\qquad
\underline u^\eta_\lambda:=v^\eta-\frac{K_0+1}{\lambda}
\]
are respectively a supersolution and a subsolution of \eqref{eq_regularized} for $\lambda\in(0,\lambda_0]$ and $\eta\in(0,\eta_0]$, provided that $\lambda_0, \eta_0$ are sufficiently small. To this end, define an elliptic operator $\mathcal{L}$ as
\begin{equation*}
	\mathcal{L}[w]:=  -\eta^{2}\Delta w -\tr \bigl(A(x) D^2w\bigr) +f(x,\lambda w) + H(x, Dw) +\lambda V. 
\end{equation*}
Since $D\overline u^\eta_\lambda=Dv^\eta$ and $D^2\overline u^\eta_\lambda=D^2 v^\eta$, it follows from \eqref{stateq1} that
\begin{align*}
	\mathcal{L}[\overline u^\eta_\lambda]-c_H= &~ f(x, \lambda v^\eta +K_0+1)+\lambda V+c_H^\eta-c_H\\
	\geq &~ f(x,1)+\lambda V+c_H^\eta-c_H=\int_0^1\partial_r f(x,r)\diff r+\lambda V+c_H^\eta-c_H
\end{align*}
where we used $\partial_rf>0$ and $f(x,0)\equiv 0$ from assumption {\rm (H3)}. Obviously, by \eqref{etaconv}
\begin{equation*}
	\lim_{\lambda\to 0, \eta\to 0}\mathcal{L}[\overline u^\eta_\lambda]-c_H>0
\end{equation*}
This means that there exist small constants $\lambda_0>0$ and $\eta_0>0$, such that $\mathcal{L}[\overline u^\eta_\lambda]\geq c_H$ for all $\lambda\in(0,\lambda_0]$ and $\eta\in(0,\eta_0]$. Similarly, we can also prove that $\mathcal{L}[\underline u^\eta_\lambda]\leq c_H$ for all $\lambda\in(0,\lambda_0]$ and $\eta\in(0,\eta_0]$. This verifies the claim.

Therefore, by Perron's method and the comparison principle, there exists a unique  solution $u_\lambda^\eta\in C^2(\T^n)$ to the regularized equation \eqref{eq_regularized} for every $\lambda\in(0,\lambda_0]$ and $\eta\in(0,\eta_0]$,
 and it satisfies 
$
\underline u^\eta_\lambda\le u^\eta_\lambda\le \overline u^\eta_\lambda.
$
Consequently, \[
\lambda \|u^\eta_\lambda\|_{L^\infty(\T^n)}
\le
\lambda \|v^\eta\|_{L^\infty(\T^n)}+K_0+1
\le
(\lambda_0+1) K_0+1
\]
uniformly for $\lambda\in(0,\lambda_0]$ and $\eta\in(0,\eta_0]$. Finally, the desired uniform gradient estimate for $\|Du_\lambda^\eta\|_{L^\infty(\T^n)}$ follows from applying Theorem \ref{The_Bernst_Lip} together with the uniform boundedness of $\lambda \|u^\eta_\lambda\|_{L^\infty(\T^n)}$.
\end{proof}

The regularized solution $u^\eta_\lambda$ of equation \eqref{eq_regularized} converges to the original solution $u_\lambda$ of \eqref{eq_intro1} as $\eta\to 0^+$. More precisely, the following lemma shows that $u^\eta_\lambda$ differs from  $u_\lambda$ by an amount proportional to $\eta/\lambda$.
\begin{Pro}[Convergence rate in the vanishing viscosity process]\label{Pro_convergerate}
Let $u_\lambda^\eta$ be the solution of \eqref{eq_regularized}, and let $u_\lambda$ be the solution of \eqref{eq_intro1}. Then
\[
\|u_\lambda^\eta-u_\lambda\|_{L^\infty(\T^n)}\le \frac{C\eta}{\lambda}
\]
for some constant $C>0$ independent of $\lambda\in(0,\lambda_0]$ and $\eta\in(0,\eta_0]$.
\end{Pro}
\begin{proof}
We give a proof based on the doubling variables method of Crandall and Lions \cite{Crandall_Lions84}. An alternative proof can be given via the nonlinear adjoint method \cite{Evans_adjoint2010, Tran_adjoint2011}.

For $\alpha>0$, we consider a function $\Phi_\alpha(x,y):\R^n\times\R^n\to\R$ defined by
\begin{equation*}
	\Phi_\alpha(x,y):=u_\lambda^\eta(x)-u_\lambda(y)-\frac{\alpha}{2}|x-y|^2,\quad \text{for all}~x, y\in\R^n.
\end{equation*}
Since $u_\lambda^\eta, u_\lambda$ are $\Z^n$-periodic, $\Phi_\alpha$ admits a global maximum $x_\alpha, y_\alpha\in [0,1]^n$  (the unit cube),
\begin{equation*}
u_\lambda^\eta(x_\alpha)-u_\lambda(y_\alpha)-\frac{\alpha}{2}|x_\alpha-y_\alpha|^2=\max_{x,y\in\R^n}\Phi_\alpha(x,y).
\end{equation*}
By the Crandall-Ishii Lemma, see \cite[Theorem 3.2]{Userguide92}, there exist matrices $X_\varepsilon,Y_\varepsilon\in\sym^n$  for every $\varepsilon>0$ such that
\[
\big(\alpha(x_\alpha-y_\alpha), X_\varepsilon\big)\in\overline{J}^{2,+}u_\lambda^\eta(x_\alpha),\qquad
\big(\alpha(x_\alpha-y_\alpha),Y_\varepsilon\big)\in\overline{J}^{2,-}u_\lambda(y_\alpha),
\]
and
\begin{equation}\label{XeYe}
\begin{pmatrix}
X_\varepsilon & 0 \\
0 & -Y_\varepsilon
\end{pmatrix}\preceq B+\varepsilon B^2,
\qquad\text{where }B:=\alpha \begin{pmatrix}
I_n & -I_n \\
-I_n & I_n
\end{pmatrix}.	
\end{equation}
Since $u_\lambda^\eta$ and $u_\lambda$ are viscosity solutions of
\eqref{eq_regularized} and \eqref{eq_intro1}, respectively, we have 
\begin{equation*}
\begin{aligned}
f(x_\alpha,\lambda u_\lambda^\eta(x_\alpha))+H(x_\alpha, \alpha(x_\alpha-y_\alpha))+\lambda V(x_\alpha) &\le  \eta^2 \tr\bigl(X_\varepsilon\bigr)+\tr\bigl(A(x_\alpha)X_\varepsilon\bigr)+ c_H\,,\\[2pt]
f(y_\alpha,\lambda u_\lambda(y_\alpha))+H(y_\alpha, \alpha(x_\alpha-y_\alpha))+\lambda V(y_\alpha) &\ge \tr\bigl(A(y_\alpha)Y_\varepsilon\bigr)+c_H\,.
\end{aligned}
\end{equation*}
Subtracting the two inequalities, we obtain
\begin{equation}\label{hlgr0}
\begin{aligned}
f(x_\alpha,\lambda u_\lambda^\eta(x_\alpha))-f(y_\alpha,\lambda u_\lambda(y_\alpha))\leq & H(y_\alpha, \alpha(x_\alpha-y_\alpha))-H(x_\alpha, \alpha(x_\alpha-y_\alpha))\\
	&+\lambda \bigl(V(y_\alpha)-V(x_\alpha)\bigr)\\
	&+\eta^2 \tr\bigl(X_\varepsilon\bigr)+\tr\bigl(A(x_\alpha)X_\varepsilon\bigr)-\tr\bigl(A(y_\alpha)Y_\varepsilon\bigr).
\end{aligned}
\end{equation}

We now estimate the right-hand side of \eqref{hlgr0}. Using $\Phi_\alpha(x_\alpha, y_\alpha)\geq \Phi_\alpha(x_\alpha, x_\alpha)$ yields 
$\alpha |x_\alpha-y_\alpha|\leq 2\Lip(u_\lambda).$ By the equi-Lipschitz property of $u_\lambda$ established in Theorem \ref{The_equiestimates}, 
\begin{equation}\label{xayabod}
	\alpha |x_\alpha-y_\alpha|\leq C_1
\end{equation}
for some constant $C_1>0$ independent of $\lambda$. Since $H\in C^2(\T^n\times\R^n)$ and $V\in \Lip(\T^n)$, 
\begin{equation*}
H(y_\alpha, \alpha(x_\alpha-y_\alpha))-H(x_\alpha, \alpha(x_\alpha-y_\alpha))
	+\lambda \bigl(V(y_\alpha)-V(x_\alpha)\bigr)\leq \frac{C_2}{\alpha}
\end{equation*}
for some constant $C_2>0$ independent of $\lambda, \eta$. Also, \eqref{XeYe} implies $X_\vep \preceq \alpha I_n+2\vep \alpha^2 I_n$, so
\begin{equation*}
	\eta^2 \tr\bigl(X_\varepsilon\bigr)\leq n(\alpha+2\vep\alpha^2)\eta^2.
\end{equation*}
Let us turn to estimate $\tr\bigl(A(x_\alpha)X_\varepsilon\bigr)-\tr\bigl(A(y_\alpha)Y_\varepsilon\bigr)$. Recall that $A(x)=\Sigma(x)^{\trans}\Sigma(x)$ by \eqref{squareofA}, we define a matrix
\[
G:=\begin{pmatrix}\Sigma(x_\alpha) & \Sigma(y_\alpha)\end{pmatrix}^\trans
      \begin{pmatrix}\Sigma(x_\alpha) & \Sigma(y_\alpha)\end{pmatrix}
    =\begin{pmatrix}
\Sigma(x_\alpha)^\trans\Sigma(x_\alpha) & \Sigma(x_\alpha)^\trans\Sigma(y_\alpha)\\
\Sigma(y_\alpha)^\trans\Sigma(x_\alpha) & \Sigma(y_\alpha)^\trans\Sigma(y_\alpha)
\end{pmatrix}\,,
\]
which is positive semidefinite. Multiplying \eqref{XeYe} by $G$ and taking the trace yields
\begin{equation*}
\tr\bigl(A(x_\alpha)X_\varepsilon\bigr)-\tr\bigl(A(y_\alpha)Y_\varepsilon\bigr)
  \le\tr(GB)+\varepsilon\,\tr(GB^2).
\end{equation*}
Here, in view of \eqref{xayabod}, 
\begin{equation*}
\tr(GB)
=
\alpha\,\tr\Bigl[\bigl(\Sigma(x_\alpha)-\Sigma(y_\alpha)\bigr)^{\trans}
\bigl(\Sigma(x_\alpha)-\Sigma(y_\alpha)\bigr)\Bigr]
\le
\alpha \,\Big(\Lip(\Sigma)|x_\alpha-y_\alpha|\Big)^2\leq \frac{ \big(C_1\Lip(\Sigma)\big)^2}{\alpha}.
\end{equation*}
Substituting the above estimates into \eqref{hlgr0} and then letting $\varepsilon\to0^+$, we obtain
\begin{equation}\label{fxlam0}
	f(x_\alpha,\lambda u_\lambda^\eta(x_\alpha))-f(y_\alpha,\lambda u_\lambda(y_\alpha))\leq C_3\left(\frac{1}{\alpha} +\alpha \eta^2\right).
\end{equation}
for some constant $C_3>0$.

Using the equiboundedness of $u_\lambda$ (see Theorem \ref{The_Bernst_Lip}) together with $\partial_rf(x,r)>0$ , the left-hand side of \eqref{fxlam0} satisfies
\begin{align}
	f(x_\alpha,\lambda u_\lambda^\eta(x_\alpha))-f(y_\alpha,\lambda u_\lambda(y_\alpha))= &f(x_\alpha,\lambda u_\lambda^\eta(x_\alpha))-f(x_\alpha,\lambda u_\lambda(y_\alpha))\nonumber\\
	&\quad+f(x_\alpha,\lambda u_\lambda(y_\alpha))-f(y_\alpha,\lambda u_\lambda(y_\alpha))\nonumber\\
	\geq &\lambda \,d_\alpha\,\big[u_\lambda^\eta(x_\alpha)-u_\lambda(y_\alpha)\big]-C_4|x_\alpha-y_\alpha|\label{jrfw11}
\end{align}
for some constant $C_4>0$, and $d_\alpha>0$ is given by
\begin{equation*}
	d_\alpha:=\int_0^1 \frac{\partial f}{\partial r}\Big(x_\alpha, \lambda \big[tu^\eta_\lambda(x_\alpha) +(1-t) u_\lambda(y_\alpha)\big]\Big)\,\diff t.
\end{equation*}
Note that by Lemma \ref{Lem_exst_ulameta} and  Theorem \ref{The_equiestimates}, $\lambda u_\lambda^\eta$ and $\lambda u_\lambda$ are uniformly bounded, so $d_\alpha\geq \kappa_0$ for some $\kappa_0>0$ independent of $\lambda, \eta$.

Now, for any $x$, we have
\begin{equation} \label{rfw22}
	u_\lambda^\eta(x_\alpha)-u_\lambda(y_\alpha)=\Phi_\alpha(x_\alpha, y_\alpha)\geq \Phi_\alpha(x,x)= u_\lambda^\eta(x)-u_\lambda(x).
\end{equation}
Combining \eqref{jrfw11}, \eqref{rfw22}, and \eqref{fxlam0}, we obtain that for any $x$,
\begin{equation*}
	\lambda \,\kappa_0\,\big[ u_\lambda^\eta(x)-u_\lambda(x)\big] \leq C_3\left(\frac{1}{\alpha} +\alpha \eta^2\right)+C_4|x_\alpha-y_\alpha|.
\end{equation*}
Choosing $\alpha=\eta^{-1}$ and using \eqref{xayabod}, we then get
\begin{equation*}
	 u_\lambda^\eta(x)-u_\lambda(x)\leq C_5\frac{\eta}{\lambda}, \quad \text{for all $x\in\T^n$}
\end{equation*}
for some constant $C_5>0$ independent of $\lambda, \eta$. By repeating the above argument, we can prove the other inequality in a similar way. 
\end{proof}

\subsection{Construction of Generalized Mather Measures via the Adjoint Method}\label{subsection_MatherMeasures}
We now use the nonlinear adjoint method introduced by Evans \cite{Evans_adjoint2010}, see also \cite{Tran_adjoint2011, Mitake_Tran2017, Tran_book2021}, to construct generalized Mather measures that will determine the limiting constraints in Theorem \ref{mainresult1}. 

We first recall the definition of generalized Mather measures in the Lagrangian formulation. Let $L:\T^n\times \R^n \to \R$ be the Lagrangian associated with the Hamiltonian $H$, i.e.,
\[ L(x,v):=\sup_{p\in\R^n} \big(p\cdot v-H(x,p)\big).\]
Due to assumption (H1) on $H$, the Lagrangian $L$ is of class $C^1$, superlinear, and strictly convex in $v$. Consider the minimization problem 
\begin{equation}\label{miniProblem}
	\inf_{\tilde\mu\in \tilde\cH}\int_{\T^n\times\R^n} L(x,v)\,\diff\tilde\mu(x,v)\,,
\end{equation}
where 
\begin{equation*}
	\tilde\cH:=\left\{\tilde\mu\in\mathcal{P}(\T^n\times\R^n):\int_{\T^n\times\R^n}  v\cdot D\phi(x)- \tr\big( A(x) D^2\phi\big) \diff\tilde\mu(x,v) = 0, ~ \forall~\phi\in C^2(\T^n)\right\}.
\end{equation*}
Here, $\mathcal{P}(\T^n\times\R^n)$ denotes the space of Radon probability measures on $\T^n\times\R^n$. Measures in $\tilde\cH$ are called \emph{holonomic} measures associated with the ergodic problem \eqref{eq_intro2}. It is known that the minimum value of \eqref{miniProblem} is equal to $-c_H$.

\begin{defn}[Generalized Mather measure]\label{generMMdef}
A generalized Mather measure $\tilde\mu$ is a minimizer of the minimization problem \eqref{miniProblem}. We denote by $\tilde\cM$ the set of all generalized Mather measures. 
\end{defn}

Generalized Mather measures for the ergodic problem \eqref{eq_intro2} were first  introduced by Gomes \cite{Gomes2002_stochasticMather, Gomes2008} in the setting of the stochastic Mather problem with $A(x)=I_n$. They can be viewed as a natural extension of the classical Mather measures \cite{Mather1991, Mane1996} associated with first-order Hamilton--Jacobi equations (i.e., when $A(x)=0$). For more general diffusion matrices $A$, see for instance \cite{Cagnetti_Gomes_Tran2011, Mitake_Tran2017, Ishii_Mitake_Tran2017_1,Mitake_Tran_uniqueness_set, Gomes_Mitake_Tran2022}. 

We shall also use the Hamiltonian representation of these measures. Let
\begin{equation}\label{legendre}
    \cL:\T^n\times\R^n\to\T^n\times\R^n,
    \qquad
    \cL(x,v):=\big(x,D_vL(x,v)\big)
\end{equation}
be the Legendre transform from Lagrangian variables to Hamiltonian variables. For a generalized Mather measure $\tilde\mu\in\tilde\cM$, we denote by
$\tilde\nu:=\mathcal L_\#\tilde\mu$ its pushforward under $\mathcal L$. This means that for every continuous and bounded test function $\psi$ on $\T^n\times\R^n$,
\begin{equation}\label{dualrela}
    \int_{\T^n\times\R^n}\psi(x,p)\,\diff\tilde\nu(x,p)
    =
    \int_{\T^n\times\R^n}\psi\big(x,D_vL(x,v)\big)\,\diff\tilde\mu(x,v).
\end{equation}
Set
\begin{equation*}
	\tilde\cM^*:=\cL_\#\tilde\cM\,.
\end{equation*}
Then every measure $\tilde\nu\in\tilde\cM^*$ satisfies the following identities:
\[\int_{\T^n\times\R^n} D_pH(x, p)\cdot p -H(x, p)\,\diff \tilde\nu(x,p) = -c_H\,,\]
and, for every $\phi\in C^2(\T^n)$,
\[\int_{\T^n\times\R^n}  D_pH(x, p)\cdot D\phi(x) - \tr\bigl( A(x) D^2\phi\bigr)\, \diff \tilde\nu(x,p) = 0.\]

\smallskip

In the present setting, we apply the nonlinear adjoint method to the regularized equation \eqref{eq_regularized}, to derive an approximation scheme for constructing generalized Mather measures.

\smallskip

\noindent\textbf{Nonlinear adjoint method.}  By Lemma \ref{Lem_exst_ulameta}, the regularized equation \eqref{eq_regularized} admits a unique $C^2$ solution $u_\lambda^\eta$ for every $\lambda\in (0,\lambda_0]$ and $\eta\in (0,\eta_0]$. This regularity allows us to introduce the associated adjoint equation.

 Fix an arbitrary point $x_0\in\mathbb{T}^n$, and let $\delta_{x_0}$ denote the Dirac delta measure. We introduce the \textbf{adjoint equation} associated with the linearization of \eqref{eq_regularized}:
\begin{equation}\label{eq_adjoint}\tag{$\rm AJ^\eta_\lambda$}
    \lambda \partial_r f(x,\lambda u_\lambda^\eta) \sigma_\lambda^\eta 
    - \operatorname{div}\bigl(D_pH(x, Du_\lambda^\eta) \sigma_\lambda^\eta\bigr)
    = \eta^{2}\Delta \sigma_\lambda^\eta 
      + \sum_{i,j=1}^{n} (a_{ij} \sigma_\lambda^\eta)_{x_i x_j}
      + \lambda \delta_{x_0} \quad \text{in}~ \mathbb{T}^n
\end{equation}
This equation admits a unique solution $\sigma_\lambda^\eta$. Owing to the singular source term $\lambda\delta_{x_0}$, the function $\sigma_\lambda^\eta$ is not smooth globally on $\T^n$, but it is of class $C^2$ on $\T^n\setminus\{x_0\}$. Therefore, $\sigma_\lambda^\eta$ is understood in the distributional sense: for every test function $\phi\in C^2(\T^n)$.
\begin{equation}\label{distribu_sense}
\int_{\T^n}
 \Big[
\lambda \partial_r f(x,\lambda u_\lambda^\eta)\phi + D_pH(x,Du_\lambda^\eta)\cdot D\phi- \eta^2 \Delta \phi- \tr\big( A(x) D^2\phi\big)\Big]\,\sigma_\lambda^\eta(x)\diff x=\lambda \phi(x_0).
\end{equation}
\,

\begin{Lem}[Properties of $\sigma_\lambda^\eta$]\label{proper_sigma}
The solution $\sigma_\lambda^\eta$ of \eqref{eq_adjoint} satisfies
 \begin{equation}\label{imprea20}
 	\sigma_\lambda^\eta(x) > 0 \quad \text{for}~x\in \T ^n\setminus\{x_0\}, \qquad \text{and}\qquad
    \int_{\T^n}\partial_r f(x,\lambda u_\lambda^\eta)\,\sigma_\lambda^\eta(x)\diff x=1.
 \end{equation}
Moreover, there exist constants $m_0>0$ and $m_1>0$, independent of $\lambda$ and $\eta$, such that
      \begin{equation}\label{imprea21}	 
      m_0\leq \int_{\T ^n} \sigma_\lambda^\eta(x)\diff x\leq m_1 \qquad \text{for all } \lambda\in (0,\lambda_0],~ \eta\in (0,\eta_0].
      \end{equation}
\end{Lem}
\begin{proof}
The positivity of $\sigma_\lambda^\eta$ on $\T^n\setminus\{x_0\}$ follows from the strong maximum principle. The integral identity follows immediately from \eqref{distribu_sense} by choosing $\phi\equiv 1$. To prove \eqref{imprea21}, we note that by Lemma \ref{Lem_exst_ulameta}, there is a constant $R>0$ such that
$
\lambda \|u_\lambda^\eta\|_{L^\infty(\T^n)}\le R
$
for all $\lambda\in(0,\lambda_0]$ and $\eta\in(0,\eta_0]$.  Since $\partial_r f>0$, we may define
\[
m_0:=\left(\max_{\substack{x\in\T^n\\ |r|\le R}}\partial_r f(x,r)\right)^{-1},
\qquad
m_1:=\left(\min_{\substack{x\in\T^n\\ |r|\le R}}\partial_r f(x,r)\right)^{-1},
\]
which are both strictly positive. Hence, \eqref{imprea20} immediately yields \eqref{imprea21}.
\end{proof}

\noindent\textbf{Construction of generalized Mather measures.} Following the techniques in \cite{Cagnetti_Gomes_Tran2011, Cagnetti_Gomes_Mitake_Tran2015, Mitake_Tran2017, Gomes_Mitake_Tran2022}, we now use the solution $u_\lambda^\eta$ of the regularized Hamilton–Jacobi equation \eqref{eq_regularized} combined with the solution $\sigma_\lambda^\eta$ of its adjoint equation \eqref{eq_adjoint} to derive a special class of generalized Mather measures.

In light of the Riesz representation theorem, for each $\lambda\in (0,\lambda_0]$ and $\eta\in (0,\eta_0]$, there exists a probability measure $\tilde\nu_\lambda^\eta\in \mathcal{P}(\T^n\times\R^n)$ satisfying 
\begin{equation}\label{def_probmea}
	\int_{\T^n\times\R^n} \psi(x,p)\diff \tilde\nu_\lambda^\eta(x,p)~=~\frac{1}{\int_{\T^n} \sigma_\lambda^\eta(x)\diff x} \int_{\T^n} \psi(x,Du_\lambda^\eta)\,\sigma_\lambda^\eta(x)\diff x
\end{equation}
for all $\psi\in C_c(\T^n\times\R^n)$. Thanks to Lemma \ref{Lem_exst_ulameta} and \eqref{imprea21} from Lemma	\ref{proper_sigma}, the family $\{\tilde\nu_\lambda^\eta\}$ is uniformly compactly supported. Consequently, \eqref{def_probmea} remains valid for every $\psi\in C(\T^n\times\R^n)$. 

By weak compactness of probability measures, there exist two subsequences $\eta_k\to 0$ and $\lambda_j\to 0$, respectively, and probability measures $\tilde\nu_{\lambda_j}, \tilde\nu\in \mathcal{P}(\T^n\times\R^n)$ such that
\begin{equation}\label{meapprosche}
\tilde\nu_{\lambda_j}^{\eta_k} \rightharpoonup \tilde\nu_{\lambda_j} \quad (\text{as}~k\to\infty), \qquad
\tilde\nu_{\lambda_j} \rightharpoonup \tilde\nu \quad (\text{as}~j\to\infty),
\end{equation}
weakly in the sense of measures in $\mathcal{P}(\T^n\times\R^n)$. 

The limiting measure $\tilde\nu$ obtained via the above approximation procedure \eqref{meapprosche} depends on the choice of $x_0\in\T^n$ in \eqref{eq_adjoint} and the subsequences $\{\eta_k\}_k, \{\lambda_j\}_j$. Accordingly, sometimes we write $\tilde\nu=\tilde\nu\left(x_0, \{\eta_k\}_k, \{\lambda_j\}_j\right)$ when it is necessary to emphasize this dependence. In general, different choices of $x_0$, $\{\eta_k\}_k$, $\{\lambda_j\}_j$ may lead to different limiting measures. We therefore denote by $\tilde\cM^*_{\rm g}$ the collection of all such  limiting measures:
\begin{equation}\label{Mgstar}
	\tilde\cM^*_{\rm g}~:=~\bigcup_{x_0, \{\eta_k\}_k, \{\lambda_j\}_j}  \tilde\nu\left(x_0, \{\eta_k\}_k, \{\lambda_j\}_j\right)\,.
\end{equation}

Accordingly, for each $\tilde\nu\in \tilde\cM^*_{\rm g}$, let $\tilde\mu={\cL}^{-1}_\#\tilde\nu$ be the pushforward of $\tilde\nu$ under the Legendre transform $\cL^{-1}$, see \eqref{legendre}. We then set
\begin{equation}\label{Mg_measures}
	\tilde\cM_{\rm g}:=\cL^{-1}_\#\tilde\cM^*_{\rm g}\,.
\end{equation}
The following result shows that $\tilde\cM_{\rm g}$ is a subset of $\tilde\cM$. 
\begin{Pro}\label{prop:minimizing}
Each measure $\tilde\mu \in \tilde\cM_{\rm g}$ is a generalized Mather measure. More precisely, since $\tilde\mu={\cL}^{-1}_\#\tilde\nu$ with $\tilde\nu\in\tilde\cM^*_{\rm g}$, we have
 \begin{enumerate}[\rm (i)]
 	  	\item $\displaystyle\int_{\T^n\times\R^n} D_pH(x, p)\cdot p -H(x, p)\,\diff \tilde\nu(x,p) = \int_{\T^n\times\R^n}L(x,v)\,\diff \tilde\mu(x,v)=-c_H\,.$
 	  	  
 	  	  \medskip
 	  	\item For every $\phi\in C^2(\T^n)$,
 	     \begin{align*}
 	     	&\int_{\T^n\times\R^n}  D_pH(x, p)\cdot D\phi(x) - \tr\big( A(x) D^2\phi\big)\, \diff \tilde\nu(x,p) \\
 	     	=\quad &\int_{\T^n\times\R^n}  v\cdot D\phi(x)-\tr\big( A(x) D^2\phi\big)\,\diff\tilde\mu(x,v)=0.
 	     \end{align*}
 \end{enumerate} 
\end{Pro}

\begin{proof}
Recall that the solution $u_\lambda^\eta$ of \eqref{eq_regularized} is $C^2$, and we can rewrite \eqref{eq_regularized} as
\begin{equation*}
\begin{aligned}
  &D_pH(x, Du^\eta_\lambda) \cdot Du^\eta_\lambda -H(x, Du^\eta_\lambda)\\
  =~ &D_pH(x, Du^\eta_\lambda)\cdot Du^\eta_\lambda-\eta^2\Delta u_\lambda^\eta-\tr\bigl( A(x) D^2u^\eta_\lambda\bigr)+f(x,\lambda u^\eta_\lambda)+\lambda V(x)-c_H.
\end{aligned}
\end{equation*}
Multiplying this by $\sigma^\eta_\lambda$, integrating over $\T^n$, and applying \eqref{distribu_sense} with $\phi=u_\lambda^\eta$, we obtain 
\begin{equation}\label{simw01}
\begin{aligned}
  &\int_{\T^n}\bigl[D_pH(x, Du^\eta_\lambda)\cdot Du^\eta_\lambda -H(x, Du^\eta_\lambda)\bigr]\,\sigma^\eta_\lambda\,\diff x\\
  =~&\lambda u_\lambda^\eta(x_0)+\int_{\T^n}\bigl[f(x,\lambda u^\eta_\lambda)-\lambda \partial_r f(x,\lambda u_\lambda^\eta ) u^\eta_\lambda+\lambda V(x)-c_H\bigr]\,\sigma_\lambda^\eta \,\diff x.
\end{aligned}
\end{equation}
Using \eqref{def_probmea}, the above identity \eqref{simw01} can be rewritten as
\begin{equation}\label{simw02}
\begin{aligned}
  &\int_{\T^n\times\R^n} D_pH(x, p)\cdot p -H(x, p)\, \diff \tilde\nu^\eta_\lambda(x,p)\\
  =~&\frac{\lambda u_\lambda^\eta(x_0) +\displaystyle{\int_{\T^n}\bigl[f(x,\lambda u^\eta_\lambda)-\lambda \partial_r f(x,\lambda u_\lambda^\eta ) u^\eta_\lambda+\lambda V(x)-c_H\bigr]\,\sigma_\lambda^\eta(x) \diff x}}{\displaystyle{\int_{\T^n} \sigma_\lambda^\eta(x)\diff x}}.
\end{aligned}
\end{equation}
The denominator $\int_{\T^n} \sigma_\lambda^\eta(x)\,\diff x\geq m_0>0$ by \eqref{imprea21} in Lemma \ref{proper_sigma}. Furthermore, by our assumption (H3), $f(x,r)\in C^1$ and $f(x,0)\equiv 0$, so  
\begin{equation}\label{eq_modulus} 
|f(x,\lambda u^\eta_\lambda)-\lambda \partial_r f(x,\lambda u_\lambda^\eta )u^\eta_\lambda| \leq \lambda\|u^\eta_\lambda\|_{\infty}\,\omega\bigl(\lambda\|u^\eta_\lambda\|_{\infty}\bigr),
\end{equation}
for some modulus of continuity $\omega$ with $\lim_{s\to 0^+}\omega(s)=0$. By Proposition \ref{Pro_convergerate} and the equiboundedness of $\{u_\lambda\}_\lambda$ established in Theorem \ref{The_equiestimates},
\begin{equation}\label{eq_bound_uel}
\lambda\|u^\eta_\lambda\|_{\infty}\leq C\eta +\lambda\|u_\lambda\|_{\infty}\leq C'(\lambda+\eta).
\end{equation}
Since $\tilde\nu=\tilde\nu\left(x_0, \{\eta_k\}_k, \{\lambda_j\}_j\right)\in \tilde\cM^*_{\rm g}$, we now set $\eta=\eta_k$ and $\lambda=\lambda_j$ in \eqref{simw02}, and let $k\to\infty$ and then $j\to\infty$. Combining this with \eqref{eq_modulus}--\eqref{eq_bound_uel}, and in view of \eqref{dualrela}, we obtain
\begin{equation*}
	  \int_{\T^n\times\R^n}D_pH(x, p)\cdot p -H(x, p)\,\diff\tilde\nu(x,p) = \int_{\T^n\times\R^n}L(x,v)\,\diff \tilde\mu(x,v)=-c_H.
\end{equation*}

We now proceed to prove item (ii). Fix a test function $\phi\in C^2(\T^n)$, by \eqref{distribu_sense} we have 
\begin{equation*}
\begin{aligned}
&\int_{\T^n}\bigl[D_pH(x, Du_\lambda^\eta)\cdot D\phi-\tr\bigl( A(x) D^2\phi\bigr)\bigr]\,\sigma_\lambda^\eta(x)\,\diff x\\
=~&\int_{\T^n}\bigl[\eta^2\Delta\phi-\lambda \partial_r f(x,\lambda u_\lambda^\eta )\phi\bigr]\,\sigma_\lambda^\eta(x)\,\diff x+\lambda \phi(x_0).
\end{aligned}
\end{equation*}
In view of \eqref{def_probmea} this becomes
\begin{equation*}
\begin{aligned}
&\int_{\T^n\times\R^n}\bigl[D_pH(x, p)\cdot D\phi-\tr\bigl( A(x) D^2\phi\bigr)\bigr] \,\diff \tilde\nu_\lambda^\eta(x,p)\\
={}&\int_{\T^n\times\R^n}\bigl[\eta^2\Delta\phi-\lambda \partial_r f(x,\lambda u_\lambda^\eta )\phi\bigr]\,\diff \tilde\nu_\lambda^\eta(x,p)+\frac{\lambda \phi(x_0)}{\displaystyle{\int_{\T^n}\sigma_\lambda^\eta(x)\,\diff x}}.
\end{aligned}
\end{equation*}
Finally, taking $\eta=\eta_k$, $\lambda=\lambda_j$ in the above identity, letting $k\to \infty$ and then $j\to \infty$, gives
\begin{equation*}
\int_{\T^n\times\R^n}\big[D_pH(x, p)\cdot D\phi(x)-\tr\big( A(x) D^2\phi\big)\big] \,\diff \tilde\nu(x,p) = 0. 
\end{equation*}
This completes the proof.
\end{proof}

%%%%%%%%%%%%%%%%%%%%%%%%%%%%%%%%%%%%%%%%%%%%%%%%%%%%%%%%%%%%%%%%%%
\section{Smooth Approximation of Lipschitz Subsolutions}\label{sec_smoothapproximation}
Since viscosity solutions of the ergodic problem \eqref{eq_intro2} are generally only Lipschitz, one needs to handle the lack of smoothness in the asymptotic analysis. More precisely, it is unclear whether every Lipschitz viscosity subsolution $w$ of the ergodic problem \eqref{eq_intro2} can be approximated by a family of smooth approximate subsolutions $\{w_\eta\}$ satisfying  
\begin{equation}
   \|w_\eta-w\|_{L^\infty}\leq O(\eta) \quad\text{and}\quad H(x, Dw_\eta)\leq \tr \big(A(x) D^2w_\eta\big)+c_H+O(\eta)\,.
\end{equation}
Such approximation results are well known in the first-order case $A(x)\equiv0$. In the second-order setting, the problem becomes significantly more delicate due to the degenerate diffusion term $\tr\big(A(x)D^2w\big)$, and this is one of the main differences between first- and second-order equations. The commutation lemma of \cite{Mitake_Tran2017} is one technical device to handle this issue in the scalar diffusion case $A(x)=a(x)I_n$. However, it is still unknown whether an analogous result holds for a general diffusion matrix. 

We therefore adopt a different approximation procedure in this section. We establish a weaker smooth approximation result, which is sufficient for the purposes of our selection problem. Let $c_0\in \R$, and suppose that $w\in \Lip(\T^n)$ is a viscosity subsolution of
\begin{equation}\label{criteq_c}
    H(x, Dw)\leq \tr \big(A(x) D^2w\big)+c_0 \qquad \text{in}~ \T^n.
\end{equation}
We construct a family of smooth approximations $\{\wreg\}_{\eta>0}$ of $w$ such that each $\wreg$ satisfies a subsolution inequality for a perturbed equation with an additional viscosity term $-\eta^2\Delta \wreg$, up to an error of order $O(\eta)$. 

\begin{The}\label{The_approx_sub}
Let $w\in \Lip(\T^n)$ be a viscosity subsolution of \eqref{criteq_c}. Then, for any $\eta \in (0,1)$, there exists a function $\wreg \in C^\infty (\T^n)$ such that
\begin{equation}\label{eq_uniform_approx}
\| \wreg - w \|_{L^\infty(\T^n)} \leq C \eta,
\end{equation}
and
\begin{equation}\label{eq_w_epsilon_reg_estimate}
-\eta^2\Delta \wreg-\tr\big(A(x)D^2 \wreg\big)+H\big(x,D\wreg\big)
\leq c_0+ C \eta \qquad \text{in}~ \T^n
\end{equation}
where $C>0$ is a constant independent of $\eta$.
\end{The}
 
The proof of Theorem \ref{The_approx_sub} is based on a modification of the approximation strategy in \cite{Gomes_Mitake_Tran2022}. Our argument uses the Lasry--Lions regularization \cite{Lasry_Lions1986} and the properties in \cite{jensen1995uniformly, Crandall1996Equivalence}. 

We first recall the required facts about sup/inf convolutions. For a function $w\in C(\T^n)$ and $\vep>0$, define
\[
w^\varepsilon(x):=\sup_{y\in\T ^n}\Bigl\{w(y)-\frac{|x-y|^2}{2\varepsilon}\Bigr\},
\qquad 
w_\varepsilon(x):=\inf_{y\in\T ^n}\Bigl\{w(y)+\frac{|x-y|^2}{2\varepsilon}\Bigr\}.\]
Then, $w^\varepsilon$ is $\frac{1}{\varepsilon}$--semiconvex (i.e., $w^\varepsilon + \frac{|x|^2}{2\varepsilon}$ is convex), while $w_\varepsilon$ is $\frac{1}{\varepsilon}$--semiconcave (i.e., $w_\varepsilon - \frac{|x|^2}{2\varepsilon}$ is concave). Moreover, $w_\vep\leq w\leq w^\vep$, and $\| w_\vep-w\|_{L^\infty}\to 0$ and $\| w^\vep-w\|_{L^\infty}\to 0$ as $\vep\to 0$. 

In particular, if $w$ is Lipschitz continuous, then so are $w^\vep$ and $w_\vep$, and
\begin{equation*}
	\|Dw^\varepsilon\|_{L^\infty(\T^n)}\leq 	\|Dw\|_{L^\infty(\T^n)},\qquad \|Dw_\varepsilon\|_{L^\infty(\T^n)}\leq 	\|Dw\|_{L^\infty(\T^n)}.
\end{equation*}

The Lasry-Lions regularization  is based on a two-parameter approximation. Let $(w^{\varepsilon+\delta})_{\delta}$ be the inf-sup convolution of $w$, that is, the function obtained by applying first the sup-convolution with parameter $\varepsilon+\delta$ and then the inf-convolution with parameter $\delta$. Then it satisfies the following property: 
\begin{Lem}\label{lem_eigenvalue}
Let $w\in C(\T^n)$ and let $\eta,\vep>0$. Then,
\begin{enumerate}[\rm (1)]
	 \item $(w^{\eta+\vep})_{\vep}$ is of class $C^{1,1},$ and for any point $x$ at which $(w^{\eta+\vep})_{\vep}$ is twice differentiable, 
	    \begin{equation}\label{eq:hessian_bounds}
	      -\frac{1}{\eta}\,I_n \preceq D^{2}(w^{\eta+\vep})_{\vep}(x) \preceq \frac{1}{\vep}\,I_n.
	     \end{equation}
	
	 \item $w \leq w^\eta \leq (w^{\eta+\vep})_{\vep}$ on $\T^n$, and $\Lip(\wee)\leq \Lip(w)$.
	 
	 \item If $(w^{\eta+\vep})_{\vep}$ is twice differentiable at $\hat{x}$ and the strict inequality 
	   \[w^\eta(\hat{x})<(w^{\eta+\vep})_{\vep}(\hat{x})\] 
	   holds, then $1/\vep$ is exactly an eigenvalue of $D^{2}(w^{\eta+\vep})_{\vep}(\hat{x})$.
\end{enumerate}
\end{Lem}
The proof of Lemma \ref{lem_eigenvalue} is available in \cite{Lasry_Lions1986} and, for the final statement, in \cite{Crandall1996Equivalence} (see also \cite{jensen1995uniformly}).

\begin{Lem}\label{lemma_LLregul}
Let $w\in \Lip(\T^n)$ be a viscosity subsolution of \eqref{criteq_c}, and write
\begin{equation}\label{wsupinf}
\wee := (w^{\eta+\vep})_{\vep}
\end{equation}
for the inf-sup convolution of $w$. Then there exist constants $K>0$ and $C>0$, independent of $\eta$ and $\vep$, such that for each $\eta\in(0,1)$ and $\vep=K\eta^3$, 
\begin{equation}\label{ineq_wee_w}
	\|\wee-w\|_{L^\infty(\T^n)}\leq C\eta,
\end{equation}
and
\begin{equation}\label{eq_approximation:w_e^d}
-\eta^2\Delta \wee
-\tr\bigl(A(x)D^2 \wee\bigr)
+H\bigl(x,D\wee\bigr)
\le c_0+C\eta
\end{equation}
holds at every point $x\in\T^n$ where $\wee$ is twice differentiable. 

\end{Lem}

\begin{proof}
By Lemma \ref{lem_eigenvalue}, the function $\wee$ given in \eqref{wsupinf} is $C^{1,1}$ and is twice differentiable almost everywhere. Let $x^{*}\in\mathbb{T}^{n}$ be a point at which $\wee$ is twice differentiable. By Lemma \ref{lem_eigenvalue} item (2), we have $\wee(x^*)\geq w^\eta(x^*)$, where $w^\eta$ is the sup-convolution of $w$. Hence, we distinguish two cases: {\rm (i)} $\wee(x^*)= w^\eta(x^*)$; {\rm (ii)} $\wee(x^*)> w^\eta(x^*)$.

\smallskip

\noindent {\bf Case (i).} $\wee(x^*)= w^\eta(x^*)$. 

Recall that $w^\eta$ is $1/\eta$--semiconvex, so it satisfies 
\[-\Delta w^\eta\leq \frac{n}{\eta}\] in $\T^n$ in the viscosity sense. Moreover, it is known (see \cite{Ishii_equivalence95, Userguide92}) that the sup-convolution $w^\eta$ is a viscosity subsolution of
\[
-\tr \bigl(A(x)D^{2}w^\eta\bigr)+H(x,Dw^\eta)\leq c_0+C_1\eta\,,
\]
where the constant $C_1$ depends only on $A$, $H$ and $\Lip(w)$. Consequently, $w^\eta$ satisfies
\begin{equation}\label{eq_visc:w^e}
-\eta^2\Delta w^\eta-\tr \bigl(A(x)D^2w^\eta\bigr)+H(x,Dw^\eta)
\leq c_0+(n+C_1)\eta
\end{equation}
in the viscosity sense. Now, for any $\beta>0$, define a quadratic function $\phi_\beta\in C^\infty(\T^n)$ by
\[
\phi_\beta(x):=\wee(x^*)+D\wee(x^*)\cdot(x-x^*)+\frac12(x-x^*)^\trans\bigl(D^2\wee(x^*)+\beta I_n \bigr)(x-x^*).
\]  
Then, $\wee-\phi_\beta$ has a local maximum at $x^*$, and therefore so does $w^\eta-\phi_\beta$. Hence, the viscosity subsolution property of $w^\eta$ for \eqref{eq_visc:w^e} yields
\[
-\eta^2\Delta\phi_\beta(x^*)-\tr\bigl(A(x^*)D^2\phi_\beta(x^*)\bigr)+H\bigl(x^*,D\phi_\beta(x^*)\bigr)\le c_0+(n+C_1)\eta.
\]  
Substituting
\[
\Delta\phi_\beta(x^*)=\Delta \wee(x^*)+n\beta,\quad
D^2\phi_\beta(x^*)=D^2\wee(x^*)+\beta I_n,\quad
D\phi_\beta(x^*)=D\wee(x^*),
\]
into the above inequality and then letting $\beta\to0^+$, we obtain \eqref{eq_approximation:w_e^d} at $x=x^*$.

\smallskip

\noindent {\bf Case (ii).} $\wee(x^*)> w^\eta(x^*)$. 

By Lemma \ref{lem_eigenvalue}, $D^2\wee(x^*)\succeq -\eta^{-1}I_n$ and $\vep^{-1}$ is an eigenvalue of $D^2 \wee(x^*)$, so 
\begin{equation}\label{eq_eigenvalue:nu}
\eta^2\Delta \wee(x^*) =\eta^2 \tr \bigl(D^2\wee(x^*)\bigr)\geq \eta^2\left(\frac{1}{\vep}-\frac{n-1}{\eta}\right)=\frac{\eta^2}{\vep}-(n-1)\eta.
\end{equation}
Moreover, 
\begin{equation}\label{Neumann_trace}
\tr \bigl(A(x^*)D^2\wee(x^*)\bigr)\geq \tr \left(-A(x^*)\,\frac{1}{\eta}I_n\right)\geq -\frac{M_1}{\eta}\,,\quad \text{where}~M_1:=\max_{x\in \mathbb{T}^n}\tr \big(A(x)\big).
\end{equation}
Since $\Lip(\wee)\leq \Lip(w)$ by Lemma \ref{lem_eigenvalue} item (2), we also have
\begin{equation}\label{HM2}
	H\bigl(x^*,D\wee(x^*)\bigr)\leq M_2, \qquad
M_2:=\max\bigl\{|H(x,p)|:~ x\in\T^n, |p|\le \Lip(w)\bigr\}+|c_0|.
\end{equation}
Combining estimates \eqref{eq_eigenvalue:nu}-\eqref{HM2}, we obtain
\begin{equation}\label{Hdjf5}
		-\eta^2\Delta \wee(x^*)-\tr  \bigl(A(x^*)D^2\wee(x^*)\bigr)
      +H\bigl(x^*,D\wee(x^*)\bigr)
\leq (n-1)\eta-\frac{\eta^2}{\vep}
      +\frac{M_1}{\eta}+M_2.
\end{equation}
Now choose $\vep=K\eta^3$ with
\begin{equation*}
	K:=\frac{1}{(n-1)+M_1+M_2-c_0}>0.
\end{equation*}
Then, \eqref{Hdjf5} reduces to
\begin{equation*}
		-\eta^2\Delta\wee(x^*)-\tr  \bigl(A(x^*)D^2\wee(x^*)\bigr)
      +H\bigl(x^*,D\wee(x^*)\bigr)
\leq c_0
\end{equation*}
for all $\eta\in (0,1)$. Thus \eqref{eq_approximation:w_e^d} also holds in this case. 

Finally, since $w$ is Lipschitz continuous, \eqref{ineq_wee_w} follows from the standard approximation properties of inf- and sup-convolutions; see \cite{Lasry_Lions1986}. This completes the proof.
\end{proof}

\,

With Lemmas \ref{lem_eigenvalue} and \ref{lemma_LLregul} at hand, we can now prove Theorem \ref{The_approx_sub}.

\begin{proof}[Proof of Theorem \ref{The_approx_sub}]
Let $\wee$ be the function given by Lemma \ref{lemma_LLregul}. We further regularize $\wee$ by convolution with a standard mollifier to obtain a smooth approximation.

Let $\theta \in C_c^\infty(\mathbb{R}^n)$ be a standard mollifier, that is, $\theta\geq 0$, $\operatorname{supp} \theta \subset B_1(0)$, and $\int_{\mathbb R^n} \theta\,\diff x = 1$. For each small $\delta > 0$, set $\theta_\delta(y):= \delta^{-n} \theta\left(\frac{y}{\delta}\right)$.  We then define a $C^\infty$ function $\wreg$ by
\begin{equation}\label{wregul}
	\wreg:=\wee * \theta_\delta,\quad \text{with}~\vep=K\eta^3~\text{and}~\delta=\eta^4,
\end{equation}
where $K$ is the constant given in Lemma \ref{lemma_LLregul}. Moreover,
\[
 (\wee * \theta_\delta)(x) = \int_{\mathbb{T}^n} \wee(x-y) \theta_\delta(y) \, \diff y\qquad \text{for all}~x\in\T^n. 
\]

We first prove \eqref{eq_uniform_approx}. Indeed, using \eqref{ineq_wee_w} and the standard properties of mollifiers, 
\begin{align*}
	\|\wreg-w\|_{L^\infty}\leq &~\|\wreg-\wee\|_{L^\infty}+\|\wee-w\|_{L^\infty}\\
	\leq & ~\delta\, \Lip(\wee)+C\eta ~\leq ~\eta^4 \Lip(w)+C\eta~\leq~ C_0\eta
\end{align*}
for some constant $C_0$ independent of $\eta$. Hence \eqref{eq_uniform_approx} follows.

\smallskip

It remains to prove \eqref{eq_w_epsilon_reg_estimate}. We divide the proof into the following steps.

\noindent {\bf Step 1.} We need to estimate the quantity 
\begin{align*}
Q(x):=~&-\eta^2\Delta \wreg(x)-\tr\bigl(A(x)D^2\wreg(x)\bigr)+H\bigl(x,D\wreg(x)\bigr)\\
    =~& -\eta^2\int_{\T^n} \Delta \wee(x-y)\,\theta_\delta(y)\,\diff y	
	-\int_{\T^n}\tr\big(A(x)D^2\wee(x-y)\big)\,\theta_\delta(y)\,\diff y\\
	&\quad +H\left(x,\int_{\T^n} D\wee(x-y)\,\theta_\delta(y)\,\diff y\right).
\end{align*}
Since $H(x,p)$ is convex in $p$, Jensen's inequality gives
\begin{align*}
	Q(x)\leq ~& -\eta^2\int_{\T^n} \Delta \wee(x-y)\,\theta_\delta(y)\,\diff y	
	-\int_{\T^n}\tr\big(A(x)D^2\wee(x-y)\big)\,\theta_\delta(y)\,\diff y\\
	&\quad +\int_{\T^n}H\big(x, D\wee(x-y)\big)\,  \theta_\delta(y) \, \diff y \\
	=~& Q_1(x)+Q_2(x)
\end{align*}
where 
\begin{align*}
	Q_1(x):=~& -\int_{\T^n}  \tr\Big( \big(A(x)-A(x-y)\big)D^2\wee(x-y)\Big)\, \theta_\delta(y)\, \diff y\\
	&~+\int_{\T^n} \left[H\big(x, D\wee(x-y)\big)-H\big(x-y, D\wee(x-y)\big)\right]\theta_\delta(y)\diff y\,,
\end{align*}
and 
\[
\begin{aligned}
	Q_2(x):=~\int \Bigl[ &-\eta^2 \Delta \wee(x-y)- \tr\big(A(x-y)D^2 \wee(x-y) \big)\\
	&\quad +H\big(x-y, D\wee(x-y)\big)\Bigr]\,\theta_\delta(y)\,\diff y. 
	\end{aligned}
\]

\noindent {\bf Step 2.}  According to Lemma \ref{lem_eigenvalue}, 
\[
-\frac{1}{\eta}I_n\preceq D^2\wee(x)\preceq \frac{1}{\varepsilon}I_n
\quad \text{for a.e. }x\in\T^n,\qquad \Lip(\wee)\le \Lip(w).
\]
Combining these estimates with the fact that $\supp(\theta_\delta)\subset B_\delta(0)$, we obtain 
\begin{equation}\label{Q1est}
 \begin{aligned}
	Q_1(x)\leq ~&  \Lip(A)\,\delta\,(\eta^{-1}+\vep^{-1})+\max_{\substack{|p|\leq \Lip(w)}} |D_xH(x,p)| \,\delta~\leq ~ C_1\eta,
 \end{aligned}
\end{equation}
where in the last step we used $\vep=K\eta^3$ and $\delta=\eta^4$ from \eqref{wregul}, and
$C_1>0$ is a constant independent of $\eta$.

\noindent {\bf Step 3.}  By \eqref{eq_approximation:w_e^d} in Lemma \ref{lemma_LLregul}, we immediately obtain 
\begin{equation}\label{Q2est}
	Q_2(x)\leq \int_{\T^n} (c_0+C_2\eta)\,\theta_\delta(y)\,\diff y=c_0+C_2\eta.
\end{equation}
for some constant $C_2>0$ independent of $\eta$. Combining \eqref{Q1est} and \eqref{Q2est}, we conclude 
\[
Q(x)\le c_0+(C_1+C_2)\eta.
\]
This finally proves the desired \eqref{eq_w_epsilon_reg_estimate}.
\end{proof}

%%%%%%%%%%%%%%%%%%%%%%%%%%%%%%%%%%%%%%%%%%%%%%%%%%%%%%%%%%%%%%%%%%
\section{{Asymptotic Analysis and Proof of the Main Results}}\label{sec_proofofmainresult}
In this section, we complete the proof of the main results. We first derive two key estimates on the asymptotic behavior of the solution $u_\lambda$ of equation \eqref{eq_intro1}. These estimates are then used to establish the uniform convergence of $u_\lambda$ and the characterization of the selected limit, thereby proving Theorem \ref{mainresult1}. Theorem \ref{mainresult2} and Theorem \ref{mainresult3} are then derived as applications of Theorem  \ref{mainresult1}.

\subsection{Key Estimates for the Selection of the Limit}
We use Theorem \ref{The_approx_sub} together with the solutions $\sigma_\lambda^\eta$ of the adjoint equations \eqref{eq_adjoint} to derive two technical lemmas. The first lemma gives an integral inequality for $u_\lambda$ with respect to the generalized Mather measures constructed in Section \ref{section_generalized_Mathermeasures}.

\begin{Lem}\label{new_gen_1}
Let $u_\lambda$ be a continuous viscosity solution of \eqref{eq_intro1}, that is,
\[
    f(x,\lambda u_\lambda)+ H(x, Du_\lambda)+\lambda V(x)=\tr \bigl( A(x) D^{2}u_\lambda\bigr)+c_H \qquad \text{in}~ \T^n.
\]
Then, 
\begin{equation}\label{frac1_00}
	\int_{\T^n\times\R^n} f(x,\lambda u_\lambda(x))\, \diff \tilde\mu  \leq  -\lambda \int_{\T^n\times\R^n} V(x)\, \diff \tilde\mu,\quad \text{for all}~\tilde\mu\in\tilde\cM_{\rm g}
\end{equation}
where $\tilde\cM_{\rm g}$ is the collection of generalized Mather measures given in \eqref{Mg_measures}. In particular, 
\begin{equation}\label{frac_00}
	\limsup_{\lambda\to 0^+}\int_{\T^n\times\R^n} \partial_r f(x,0)\, u_\lambda(x) \diff \tilde\mu  \leq  -\int_{\T^n\times\R^n} V(x) \diff \tilde\mu,\quad \text{for all}~\tilde\mu\in\tilde\cM_{\rm g}.
\end{equation}

\end{Lem}

\begin{proof}
Fix $\lambda>0$. We first prove inequality \eqref{frac1_00}. Note that $u_{\lambda}$ is Lipschitz continuous by Theorem \ref{The_Bernst_Lip}. Set
\begin{equation*}
P(x):=f(x,\lambda u_{\lambda}(x))+\lambda V(x).	
\end{equation*}
For each $\vep>0$, by a standard mollification argument there exists $P_\vep\in C^\infty(\T^n)$ such that
\begin{equation}\label{Pvep}
	\|P_\vep-P\|_{L^\infty(\T^n)}\leq \vep,\qquad \Lip(P_\vep)\leq \Lip(P).
\end{equation}
We then introduce the Hamiltonian $\HH_\vep: \T^n\times\R^n\to\R$ defined by
	\begin{equation}\label{newHam}
		\HH_\vep(x,p):=H(x,p)+P_\vep(x).
	\end{equation}
Note that $\HH_\vep\in C^2(\T^n\times\R^n)$ still satisfies assumption (H1). Clearly, $u_{\lambda}$ is a Lipschitz viscosity subsolution of
\[
    \HH_\vep(x, Du_{\lambda})\leq \tr \bigl( A(x) D^{2}u_{\lambda}\bigr)+c_H+\vep \quad \text{in}~\T^n.
\]

Fix $\tilde\mu\in\tilde\cM_{\rm g}$. Let $\tilde\nu:=\cL_\#\tilde\mu\in \tilde\cM^*_{\rm g}$ be its pushforward under the Legendre transform $\mathcal L$ in \eqref{legendre}. By \eqref{Mgstar}, there exist subsequences $\{\eta_k\}_k\to 0^+$ and  $\{\lambda_l\}_l\to 0^+$ such that 
\[\tilde\nu=\tilde\nu\left(x_0, \{\eta_k\}_k, \{\lambda_l\}_l\right)\] 
with associated measures $\tilde\nu_{\lambda_l}^{\eta_k}$, $\tilde\nu_{\lambda_l}$ satisfying 
\[
\tilde\nu_{\lambda_l}^{\eta_k} \rightharpoonup \tilde\nu_{\lambda_l} \quad (k\to\infty), \qquad
\tilde\nu_{\lambda_l} \rightharpoonup \tilde\nu \quad (l\to\infty).
\]
Moreover, $\tilde\nu_{\lambda_l}^{\eta_k}$ is compactly supported and satisfies
\begin{equation}\label{testmulk}
	\int_{\T^n\times\R^n} \psi(x,p)\,\diff \tilde\nu_{\lambda_l}^{\eta_k}~=~\frac{1}{\int_{\T^n} \sigma_{\lambda_l}^{\eta_k}(x)\,\diff x} \int_{\T^n} \psi\left(x,Du_{\lambda_l}^{\eta_k}\right)\sigma_{\lambda_l}^{\eta_k}(x)\,\diff x
\end{equation}
for all $\psi\in C(\T^n\times\R^n)$, where $\sigma_{\lambda}^{\eta}$ is the solution of the adjoint equation \eqref{eq_adjoint}.

Applying the approximation Theorem \ref{The_approx_sub} with $\eta=\eta_k$, Hamiltonian $\HH_\vep$ and $c_0=c_H+\vep$, we obtain $\phi_{k}\in C^\infty(\T^n)$ such that 
\[\|\phi_{k}-u_{\lambda}\|_{L^\infty(\T^n)}\leq C\eta_k\] and
\begin{equation}\label{ineq_etak}
	-\eta_k^2\Delta \phi_k-\textup{tr}\big( A(x) D^2\phi_k\big)+\HH_\vep(x,D\phi_k)\leq c_H+\vep+C\eta_k\quad \text{for all}~x\in\T^n,
\end{equation}
for some constant $C>0$ independent of $k$. Integrating \eqref{ineq_etak} with respect to the measure $\tilde\nu_{\lambda_l}^{\eta_k}$ and using $\HH_\vep(x,D\phi_{k})$ $\geq$ $\HH_\vep(x,p)+D_p\HH_\vep(x, p)\cdot (D\phi_k-p)$ from the $p$-convexity of $\HH_\vep(x,p)$, we get  
\begin{align*}
c_H+\vep+C\eta_k \geq & \int_{\T^n\times\R^n}\HH_\vep(x,D\phi_k)-\eta_k^2\Delta \phi_k-\textup{tr}\big( A(x) D^2\phi_k\big)\,\diff \tilde\nu_{\lambda_l}^{\eta_k}\\
		\geq & \int_{\T^n\times\R^n} \HH_\vep(x,p)-D_p\HH_\vep(x, p)\cdot p \,\diff \tilde\nu_{\lambda_l}^{\eta_k}\\
		  &\qquad +\int_{\T^n\times\R^n} D_p\HH_\vep(x, p)\cdot D\phi_{k}-\eta_k^2\Delta \phi_k-\textup{tr}\big( A(x) D^2\phi_k\big)\,\diff \tilde\nu_{\lambda_l}^{\eta_k}\,.
\end{align*}
To estimate the last integral term above, using \eqref{testmulk} and then applying identity \eqref{distribu_sense} with test function $\phi=\phi_k\in C^2(\T^n)$, we obtain
\begin{align*}
	&\int_{\T^n\times\R^n} D_p\HH_\vep(x, p)\cdot D\phi_k-\eta_k^2\Delta \phi_k-\textup{tr}\big( A(x) D^2\phi_k\big)\,\diff \tilde\nu_{\lambda_l}^{\eta_k}\\
			= &~ \frac{\displaystyle\int_{\T^n}\left[D_pH(x, Du_{\lambda_l}^{\eta_k})\cdot D\phi_k -\eta_k^2\Delta \phi_k-\textup{tr}\big( A(x) D^2\phi_k\big)\right]\sigma_{\lambda_l}^{\eta_k}\,\diff x }{\displaystyle\int_{\T^n} \sigma_{\lambda_l}^{\eta_k}\,\diff x}\\
		=&~ \frac{\lambda_l \phi_k(x_0)- \lambda_l \displaystyle\int_{\T^n}\partial_rf(x,\lambda_l u_{\lambda_l}^{\eta_k}) \,\phi_k\, \sigma_{\lambda_l}^{\eta_k} \, \diff x }{\displaystyle\int_{\T^n} \sigma_{\lambda_l}^{\eta_k}\,\diff x} \,.
\end{align*}
Combining this with the preceding inequality, we get
\begin{align*}
		&c_H+\vep+C\eta_k\\
		\geq &~ \int \HH_\vep(x,p)-D_p\HH_\vep(x, p)\cdot p \diff \tilde\nu_{\lambda_l}^{\eta_k}
		 + \frac{\lambda_l \phi_k(x_0)- \lambda_l\displaystyle\int_{\T^n}\partial_rf(x,\lambda_l u_{\lambda_l}^{\eta_k}) \,\phi_k\,\sigma_{\lambda_l}^{\eta_k}  \, \diff x }{\displaystyle\int_{\T^n} \sigma_{\lambda_l}^{\eta_k}\diff x}.  
\end{align*}
By Lemma \ref{Lem_exst_ulameta} and Lemma \ref{proper_sigma}, both $\lambda_l u_{\lambda_l}^{\eta_k}$ and $\int\sigma_{\lambda_l}^{\eta_k}\diff x$ are uniformly bounded. Therefore, letting $k\to \infty$ and then $l\to \infty$, and using $\|\phi_k-u_\lambda\|_{L^\infty}\to 0$, we obtain
\begin{equation*}
c_H+\vep \geq   \int_{\T^n\times\R^n} \HH_\vep(x,p)-D_p\HH_\vep(x, p)\cdot p \,\diff \tilde\nu.
\end{equation*}
By \eqref{newHam} and item (i) of Proposition \ref{prop:minimizing}, this becomes
\begin{equation*}
c_H+\vep\geq \int_{\T^n\times\R^n}H(x,p)+P_\vep(x)-D_pH(x,p)\cdot p\,\diff\tilde\nu=c_H+ \int_{\T^n\times\R^n}P_\vep(x)\,\diff\tilde\nu.
\end{equation*}
Sending $\vep\to 0$ and using \eqref{Pvep} yields 
\begin{equation*}
	\int_{\T^n\times\R^n} P(x)\,\diff \tilde\nu=\int_{\T^n\times\R^n} P(x)\,\diff \tilde\mu\leq 0.
\end{equation*}
Thus, by the definition of $P$,  the inequality \eqref{frac1_00} is proved.

To prove inequality \eqref{frac_00}, note that $\lambda u_\lambda\to 0$ uniformly as $\lambda\to 0^+$ by Theorem \ref{The_equiestimates}. Using assumption (H3), we have
\[f(x,\lambda u_\lambda(x))=\lambda \partial_r f(x,0)u_{\lambda}(x)+o(\lambda),\quad\text{uniformly in}~ x\in\T^n.\]
 Then inequality \eqref{frac1_00} becomes 
\[\int_{\T^n\times\R^n}\lambda \partial_r f(x,0)u_{\lambda}(x) \diff\tilde\mu+o(\lambda)\leq -\lambda \int_{\T^n\times\R^n}V(x)\diff \tilde\mu.
\]
Dividing both sides by $\lambda$ and taking the $\limsup$ as $\lambda\to 0^+$, we obtain the desired inequality \eqref{frac_00}. 
\end{proof}

In contrast to Lemma \ref{new_gen_1}, the next lemma establishes a lower bound for the $C^2$ solution $u_\lambda^\eta$ of the regularized equation \eqref{eq_regularized}, rather than for $u_\lambda$ itself, since its proof relies on the adjoint method and the $C^2$ regularity. Thanks to the uniform convergence $u_\lambda^\eta \to u_\lambda$ (see Proposition \ref{Pro_convergerate}), the estimate obtained for $u_\lambda^\eta$ will later serve as a bridge to recover the asymptotic information on $u_\lambda$.

\begin{Lem}\label{new_gen_2}
Let $u_\lambda^\eta$ be the solution to the regularized equation \eqref{eq_regularized}, that is,
\begin{equation}\label{ejqiu01}
f(x,\lambda u_\lambda^\eta)+H(x,Du_\lambda^\eta)+\lambda V(x)
=
\eta^2\Delta u_\lambda^\eta+\tr\big(A(x)D^2u_\lambda^\eta\big)+c_H
\quad \text{in}~\T^n,
\end{equation}
with parameters $\lambda\in(0,\lambda_0]$ and $\eta\in(0,\eta_0]$. Then, for every continuous viscosity solution $w$ of equation \eqref{eq_intro2} and every point $x_0\in \mathbb{T}^n$, the following estimate holds:
\begin{equation}\label{eqwe38}
u_\lambda^\eta (x_0)\geq w(x_0)-\int_{\mathbb{T}^n}\big(\partial_rf(x,0)\,w+V\big)\sigma_\lambda^\eta\diff x
-C\left[\left(\frac{\lambda+\eta}{\lambda}+\|w\|_\infty\right)\breve{\omega}(\lambda+\eta)+ \frac{\eta}{\lambda}\right]
\end{equation}
where $\sigma_\lambda^\eta$ is the solution of the adjoint equation \eqref{eq_adjoint}, $\breve{\omega}$ is a modulus of continuity, and $C>0$ is a constant independent of $\lambda$ and $\eta$.
\end{Lem}

\begin{proof}
Since $w$ is Lipschitz continuous by Theorem \ref{The_Bernst_Lip}, we may invoke the approximation Theorem \ref{The_approx_sub}. Specifically, for $\eta\in (0,\eta_0]$,  there exists a function $\wreg\in C^\infty(\mathbb{T}^n)$ such that
\begin{equation}\label{eq:uniform_approx_applying}
\|\wreg-w\|_{L^\infty(\T^n)}\leq C_1 \eta,
\end{equation}
\begin{equation}\label{eq:approx_sub_applying}
-\eta^2\Delta \wreg-\operatorname{tr}\big(A(x) D^2 \wreg\big)+H(x, D\wreg)\leq c_H+ C_1 \eta \qquad\text{in}~\mathbb{T}^n, 
\end{equation}
where $C_1>0$ is a constant independent of $\eta$.

Subtracting the regularized equation \eqref{ejqiu01} from \eqref{eq:approx_sub_applying}, we deduce that 
\begin{equation*}
\begin{aligned}
	& ~f(x,\lambda u_\lambda^\eta)+\lambda V +C_1 \eta\\
\ge & ~ H(x, D\wreg)-H(x, Du^\eta_\lambda)
-\eta^2\Delta(\wreg-u_\lambda^\eta)-\operatorname{tr}\left(A(x) D^2\big(\wreg-u_\lambda^\eta\big)\right)\\
\geq & ~ D_pH(x,Du_\lambda^\eta)\cdot D(\wreg-u_\lambda^\eta)-\eta^2\Delta(\wreg-u_\lambda^\eta)-\operatorname{tr}\left(A(x) D^2\big(\wreg-u_\lambda^\eta\big)\right),
\end{aligned}	
\end{equation*}
where the last inequality follows from the $p$-convexity of $H(x,p)$. Then, multiplying this inequality by $\sigma_\lambda^\eta\geq 0$ and integrating over $\T^n$, we find 
\begin{equation*}
\begin{aligned}
	&\int_{\T^n} \left[f(x,\lambda u_\lambda^\eta)+\lambda V +C_1 \eta\right]\sigma_\lambda^\eta\diff x\\
\geq & \int_{\T^n}   \left[D_pH(x,Du_\lambda^\eta)\cdot D(\wreg-u_\lambda^\eta)-\eta^2\Delta(\wreg-u_\lambda^\eta)-\operatorname{tr}\left(A(x) D^2\big(\wreg-u_\lambda^\eta\big)\right)\right]\sigma_\lambda^\eta\diff x.
\end{aligned}
\end{equation*}
Applying identity \eqref{distribu_sense} with test function $\phi=\wreg-u_\lambda^\eta\in C^2(\T^n)$, we obtain
\begin{equation}\label{ajkws1}
\begin{aligned}
    &~ \int_{\T^n} \left[f(x,\lambda u_\lambda^\eta)+\lambda V +C_1 \eta\right]\sigma_\lambda^\eta\diff x\\
\ge &~ \lambda\,\big(\wreg(x_0)-u_\lambda^\eta(x_0)\big)-\lambda\int_{\T^n}\partial_rf(x,\lambda u_\lambda^\eta)\,(\wreg-u_\lambda^\eta)\,\sigma_\lambda^\eta\diff x.
\end{aligned}
\end{equation}
Using estimate \eqref{eq:uniform_approx_applying} together with the following identity (see \eqref{imprea20} of Lemma \ref{proper_sigma})
\begin{equation*}
	\int_{\T^n}\partial_rf(x,\lambda u_\lambda^\eta)\,\sigma_\lambda^\eta\,\diff x=1,
\end{equation*}
the above inequality \eqref{ajkws1} becomes
\begin{equation*}
\begin{aligned}
    &~\int_{\T^n} \left[f(x,\lambda u_\lambda^\eta)+\lambda V +C_1 \eta\right]\sigma_\lambda^\eta\diff x\\
\ge &~ \lambda\,\big(w(x_0)-u_\lambda^\eta(x_0)\big)-\lambda\int_{\T^n}\partial_rf(x,\lambda u_\lambda^\eta)\,(w-u_\lambda^\eta)\,\sigma_\lambda^\eta\diff x-C_1\lambda \eta.
\end{aligned}
\end{equation*}
Dividing by $\lambda>0$ and rearranging terms, we arrive at 
\begin{equation}\label{daskjoi2}
	u_\lambda^\eta(x_0)\ge w(x_0)-\int_{\T^n}\big(\partial_rf(x,0)\,w+V\big)\,\sigma_\lambda^\eta\diff x
-\int_{\mathbb{T}^n}R(x,\lambda,\eta)\,\sigma_\lambda^\eta\diff x-C_1\eta
\end{equation}
where 
\begin{align*}
	R(x,\lambda,\eta):=\frac{\displaystyle{ f(x,\lambda u_\lambda^\eta)-\lambda\partial_rf(x,\lambda u_\lambda^\eta)u_\lambda^\eta}}{\lambda} +\Big(\partial_r f(x,\lambda u_\lambda^\eta)-\partial_r f(x,0)\Big)\,w+C_1\frac{\eta}{\lambda}.
\end{align*}

It remains to estimate $R(x,\lambda,\eta)$. Since $f\in C^1$ and $f(x,0)\equiv 0$ by assumption (H3), there exists a modulus of continuity $\breve\omega_1$ such that $|f(x,r)-\partial_r f(x,0)\,r|\leq |r|\,\breve\omega_1(|r|)$, which implies 
\begin{equation}\label{Rest1}
	\left|\frac{f(x,\lambda u_\lambda^\eta)-\lambda\partial_rf(x,\lambda u_\lambda^\eta)u_\lambda^\eta}{\lambda}\right| \leq \|u^\eta_\lambda\|_{\infty}\,\breve\omega_1\big(\lambda\|u^\eta_\lambda\|_{\infty}\big).
\end{equation}
Since $\partial_r f(x,r)$ converges uniformly to $\partial_rf(x,0)$ as $r\to 0$, there exists another modulus of continuity $\breve\omega_2$ such that
\begin{equation}\label{Rest2}
    \Big|\Big(\partial_r f(x,\lambda u_\lambda^\eta)-\partial_r f(x,0)\Big)\,w\Big|\leq \|w\|_\infty\,\breve\omega_2\big(\lambda\|u^\eta_\lambda\|_{\infty}\big).
\end{equation}
On the other hand, Proposition \ref{Pro_convergerate} together with the uniform boundedness of $\{u_\lambda\}_\lambda$ (see Theorem \ref{The_equiestimates}) implies 
\begin{equation}\label{Rest3}
   \lambda\|u^\eta_\lambda\|_{\infty}\leq C_2\,(\lambda+\eta)
\end{equation}
for some constant $C_2$ independent of $\lambda, \eta$. Hence, by \eqref{Rest1}--\eqref{Rest3} and Lemma \ref{proper_sigma}, we obtain 
\begin{equation}\label{R_estim}
\left|\int_{\T^n}R(x,\lambda,\eta)\,\sigma_\lambda^\eta\,\diff x\right| \leq C_3\left[\left(\frac{\lambda+\eta}{\lambda}+\|w\|_\infty\right)\breve\omega(\lambda+\eta)\,+\frac{\eta}{\lambda}\right],
\end{equation}
for some constant $C_3>0$ independent of $\lambda$ and $\eta$, and some modulus of continuity $\breve\omega$ satisfying $\lim_{s\to 0^+}\breve\omega(s)=0$.

Finally, substituting estimate \eqref{R_estim} back into the previous lower bound \eqref{daskjoi2} for $u_\lambda^\eta(x_0)$, we obtain the desired estimate \eqref{eqwe38}.
\end{proof}

\subsection{Selection of the Limit}
 Lemma \ref{new_gen_1} shows that any possible limit of the family $\{u_\lambda\}_{\lambda>0}$ must satisfy certain integral inequalities. This naturally leads us to consider the class $\calE$ of viscosity solutions of the ergodic problem that satisfy these constraints, and to define the candidate limit as the maximal element of this class.
 
 More precisely, let $\calE$ denote the family of all continuous viscosity solutions $w:\T^n\to\R$ of the ergodic problem \eqref{eq_intro2} such that
\begin{equation*}
\int_{\T^n\times\R^n} \partial_r f(x,0)\, w(x) \diff \tilde\mu \leq  -\int_{\T^n\times\R^n}V(x) \diff \tilde\mu\,,\quad\text{for all}~\tilde\mu\in \tilde\cM_{\rm g}.
\end{equation*}
We then define a function $u_*: \T^n\to\R$ by
\begin{equation}\label{u0defin}
	u_*(x):=~\sup_{w\in \calE} w(x).
\end{equation}

We are now ready to show that $u_*$ is precisely the selected limit of $u_\lambda$ as $\lambda\to 0^+$.
\begin{The}\label{The_uconverg}
Let  $u_\lambda$ be the continuous viscosity solution of \eqref{eq_intro1} for $\lambda\in (0,\lambda_0]$. Then 
\begin{equation*}
	\lim_{\lambda\to 0^+} u_\lambda(x)=u_*(x),\quad\text{uniformly for}~x\in\T^n,
\end{equation*}
where $u_*$ is defined in \eqref{u0defin}. Moreover, $u_*: \T^n\to \R$ is a viscosity solution of \eqref{eq_intro2}.
\end{The}
\begin{proof}
Recall that the family $\{u_\lambda\}_{\lambda\in (0,\lambda_0]}$ is equibounded and equi-Lipschitz by Theorem \ref{The_equiestimates}. Hence, by the Ascoli--Arzel\'a theorem, it is relatively compact in $C(\T^n)$. Moreover, by the stability of viscosity solutions, every convergent subsequence of $\{u_\lambda\}$ converges to a viscosity solution of the ergodic problem \eqref{eq_intro2}. 

We claim that all convergent subsequences have the same limit.

Let $\{\lambda_j\}_{j}$ be a subsequence converging to $0$ such that $u_{\lambda_j}$ converges uniformly to some function $u$. By Lemma \ref{new_gen_1}, $u$ satisfies
 \begin{equation*}
\int_{\T^n\times\R^n} \partial_r f(x,0)\, u(x) \diff \tilde\mu \leq  -\int_{\T^n\times\R^n}V(x) \diff \tilde\mu\,,\quad\text{for all}~\tilde\mu\in \tilde\cM_{\rm g}.
\end{equation*}
Thus, $u\in \calE$, and therefore
\begin{equation}\label{jdfp}
u \leq u_*~\text{on}~\T^n.
\end{equation}

It remains to prove the reverse inequality $ u(x_0) \geq u_*(x_0)$ for each point $x_0\in \T^n$.  Choose a subsequence $\{\eta_k\}_k\to 0^+$. Let $\sigma_{\lambda_j}^{\eta_k}$ be the solution of the adjoint equation \eqref{eq_adjoint} corresponding to $\lambda=\lambda_j$ and $\eta=\eta_k$. By \eqref{def_probmea}, there exists a probability measure $\tilde\nu_{\lambda_j}^{\eta_k}\in \mathcal{P}(\T^n\times\R^n)$ such that
\begin{equation}\label{measnu_jk}
	\int_{\T^n\times\R^n} \psi(x,p)\diff \tilde\nu_{\lambda_j}^{\eta_k}~=~\frac{\displaystyle{\int_{\T^n} \psi(x,Du_{\lambda_j}^{\eta_k})\,\sigma_{\lambda_j}^{\eta_k}(x)\diff x}}{\displaystyle{\int_{\T^n} \sigma_{\lambda_j}^{\eta_k}(x)\diff x}} 
\end{equation}
for all $\psi\in C(\T^n\times\R^n)$. 

By Lemma \ref{Lem_exst_ulameta} and the estimate \eqref{imprea21} in Lemma	\ref{proper_sigma}, the family $\tilde\nu_{\lambda_j}^{\eta_k}$ is uniformly compactly supported. Hence, by weak compactness of probability measures, after passing to subsequences if necessary, we may assume that  
\begin{equation*}
\tilde\nu_{\lambda_j}^{\eta_k} \rightharpoonup \tilde\nu_{\lambda_j} \quad (\text{as}~k\to\infty), \qquad
\tilde\nu_{\lambda_j} \rightharpoonup \tilde\nu \quad (\text{as}~j\to\infty),
\end{equation*}
for some $\tilde\nu_{\lambda_j}, \tilde\nu\in\mathcal{P}(\T^n\times\R^n)$. By the definition \eqref{Mgstar}, the limiting measure $\tilde\nu\in \tilde\cM^*_{\rm g}$.

Now, fix an arbitrary $w\in \calE$. Applying Lemma \ref{new_gen_2} with $\lambda=\lambda_j$ and $\eta=\eta_k$, and using \eqref{measnu_jk}, we obtain 
\begin{equation}\label{ulejk0}
\begin{aligned}
u_{\lambda_j}^{\eta_k}(x_0)
\geq & w(x_0)-\int_{\mathbb{T}^n}\big(\partial_rf(x,0)\,w+V\big)\sigma_{\lambda_j}^{\eta_k}\diff x
-C\left[\left(1+\frac{\eta_k}{\lambda_j}+\|w\|_\infty\right)\breve{\omega}(\lambda_j+\eta_k)+\frac{\eta_k}{\lambda_j}\right]\\
=&~ w(x_0)-\frac{\displaystyle{\int \big(\partial_rf(x,0)\,w+V\big)\,\diff\tilde\nu_{\lambda_j}^{\eta_k}}}{\displaystyle{\int \partial_r f(x,\lambda_j u_{\lambda_j}^{\eta_k})\, \diff \tilde\nu_{\lambda_j}^{\eta_k}}}
-C\left[\left(1+\frac{\eta_k}{\lambda_j}+\|w\|_\infty\right)\breve{\omega}(\lambda_j+\eta_k)+\frac{\eta_k}{\lambda_j}\right].
\end{aligned}
\end{equation}
Here, the last equality used the identity $
	\int_{\T^n}\partial_rf(x,\lambda_j u_{\lambda_j}^{\eta_k})\,\sigma_{\lambda_j}^{\eta_k}\,\diff x=1$, see Lemma \ref{proper_sigma}.

According to Proposition \ref{Pro_convergerate}, we have $u_{\lambda_j}^{\eta_k}\to u_{\lambda_j}$ uniformly as $k\to \infty$. Therefore, letting $k\to \infty$ and then $j\to \infty$ in \eqref{ulejk0}, we get
\begin{equation*}
u(x_0)\geq w(x_0)-\frac{\displaystyle{\int \big(\partial_rf(x,0)\,w+V\big)\,\diff\tilde\nu}}{\displaystyle{\int \partial_r f(x,0)\,\diff \tilde\nu}}=w(x_0)-\frac{\displaystyle{\int \big(\partial_rf(x,0)\,w+V\big)\,\diff\tilde\mu}}{\displaystyle{\int \partial_r f(x,0)\,\diff \tilde\mu}},
\end{equation*}
where $\tilde\mu=\cL^{-1}_\#\tilde\nu$ is the pushforward of $\tilde\nu$ under the Legendre transform. In particular, $\tilde\mu\in \tilde\cM_{\rm g}$. Thus, we infer $u(x_0)\geq w(x_0)$ since $w\in \calE$. As $w\in\calE$ is arbitrary, it follows that 
\[u(x_0)\ge u_*(x_0).\]
Moreover, by the arbitrariness of $x_0\in\T^n$, we conclude that $u\ge u_*$ on $\T^n$. Combining this with the previous inequality \eqref{jdfp}, we obtain $u=u_*$. 

Thus every convergent subsequence of $\{u_\lambda\}$ has the same limit $u_*$, and therefore the whole family converges uniformly to $u_*$. This completes the proof.
\end{proof}

As a consequence, Theorem \ref{mainresult1} follows immediately from Theorem \ref{The_equiestimates} and Theorem \ref{The_uconverg}.

\subsection{Proof of the Comparison Result}
We prove the comparison result stated in Theorem \ref{mainresult2}. As explained in the introduction, it is an application of the selection principle established above. The argument relies on a suitable discounted approximation and the characterization of the selected limit obtained in Theorem \ref{mainresult1}.

\begin{proof}[Proof of Theorem \ref{mainresult2}]
Since $\sigma$ and $u_2$ are continuous on $\T^n$, standard mollification arguments yield two families of functions $\{\sigma^\delta\}_{\delta>0}\subset C^1(\T^n)$ and $\{V^\delta\}_{\delta>0}\subset C^1(\T^n)$ such that 
\begin{equation}\label{twoapprox}
	\| \sigma-\sigma^\delta\|_{L^\infty(\T^n)} \leq \delta,\qquad \| u_2\sigma-V^\delta\|_{L^\infty(\T^n)} \leq \delta.
\end{equation}
Since $\sigma$ is strictly positive, we may assume without loss of generality that $\sigma^\delta>0$ on $\T^n$ for every $\delta>0$.

Fix $\delta>0$. We consider the following discounted approximation of \eqref{eq_intro2}:
\begin{equation}\label{sig_delt_appeq}
	\lambda \sigma^\delta(x)u_\lambda^\delta(x)+H(x, Du_\lambda^\delta)-\lambda V^\delta(x)=\tr \bigl( A(x) D^{2}u_\lambda^\delta \bigr)+c_H\qquad \text{in}~ \T^n.
\end{equation}
This equation satisfies assumptions {\rm (H1)--(H3)}. Thus, by applying Theorem \ref{mainresult1} with $f(x,r):=\sigma^\delta(x)r$ and potential $V(x):=-V^\delta$, there is a function $u_0^\delta: \T^n \to\R$ such that 
\begin{equation}\label{limlamdelt}
	\lim_{\lambda\to 0^+}u_\lambda^\delta(x)=u^\delta_0(x)=\sup_{w\in \calE_\delta} w(x),
\end{equation}
where $\calE_\delta$ denotes the family of viscosity solutions $w$ of \eqref{eq_intro2} satisfying 
\begin{equation*}
\int_{\T^n\times\R^n} w \, \sigma^\delta \,\diff\tilde\mu \leq  \int_{\T^n\times\R^n}V^\delta \,\diff \tilde\mu \quad\text{for all}~\tilde\mu\in \tilde\cM_{\rm g}^\delta.
\end{equation*}
Here, $\tilde\cM_{\rm g}^\delta\subseteq \tilde\cM$ denotes the family of generalized
Mather measures obtained by applying the construction in Subsection
\ref{subsection_MatherMeasures} to the special case \eqref{sig_delt_appeq}.

Since $\sigma$ is positive, the quantity $m_\sigma:=\inf_{x\in\T^n}\sigma(x)$ is strictly positive. For sufficiently small $\delta>0$, we claim that 
\[u_1-C_\delta\in \calE_\delta  \quad\text{with}\quad   C_\delta:=\frac{\delta}{m_\sigma-\delta}\big(1+\|u_1\|_{L^\infty(\T^n)}\big).\] 
Indeed, since every measure $\tilde\mu\in \tilde\cM_{\rm g}^\delta$ belongs to $\tilde\cM$, by assumption we have
\begin{equation*}
\int u_1 \, \sigma\, \diff\tilde\mu \leq  \int u_2\, \sigma\, \diff \tilde\mu.
\end{equation*}
This, together with estimate \eqref{twoapprox}, yields 
\begin{align*}
	\int (u_1-C_\delta)\,\sigma^\delta \,\diff\tilde\mu 
	\leq & ~\int u_1\,\sigma^\delta \, \diff\tilde\mu-C_\delta(m_\sigma -\delta)\\
	\leq & ~\int u_1\,\sigma \, \diff\tilde\mu+\delta\|u_1\|_{L^\infty(\T^n)}  -C_\delta(m_\sigma -\delta)\\
	\leq & ~ \int u_2\,\sigma\, \diff\tilde\mu+\delta\|u_1\|_{L^\infty(\T^n)}  -C_\delta(m_\sigma -\delta)\\
	\leq & ~\int V^\delta \, \diff\tilde\mu \,.
\end{align*}
This verifies the claim $u_1-C_\delta\in \calE_\delta$. As a consequence, by \eqref{limlamdelt},  
\begin{equation}\label{upperu0delta}
u_0^\delta(x)\geq u_1(x)-C_\delta \quad \text{for all } x\in\T^n.	
\end{equation}

Next, define 
\[u^{\delta}_+(x):=u_2(x)+D_\delta \quad \text{with}\quad D_\delta:=\frac{\delta}{m_\sigma -\delta}\big(1+\|u_2\|_{L^\infty(\T^n)}\big).\] A direct verification shows that $u^{\delta}_+$ is a viscosity supersolution of equation \eqref{sig_delt_appeq} provided that $\delta$ is suitably small. By the comparison principle, $u^{\delta}_+(x)\geq u_\lambda^\delta(x)$ for all $x\in \T^n$. Letting $\lambda\to 0^+$ and using \eqref{limlamdelt}, we get  $u^{\delta}_+\geq u_0^\delta$ on $\T^n$. Hence, by \eqref{upperu0delta},
\begin{equation*}
u_2(x)+D_\delta=u^{\delta}_+\geq u_0^\delta\geq u_1(x)-C_\delta \quad \text{for all } x\in\T^n.	
\end{equation*}
Finally, since $C_\delta, D_\delta\to 0$ as $\delta\to 0^+$, we conclude that $u_1\leq u_2$ on $\T^n$. This completes the proof.
\end{proof}

\subsection{Proof of Theorem \ref{mainresult3}}

\begin{proof}[Proof of Theorem \ref{mainresult3}]
Since $\hat{u}_0\in C(\T^n)$ is a solution of the ergodic problem \eqref{eq_intro2}, the Lipschitz estimate (see for instance Theorem \ref{The_Bernst_Lip}) yields $\hat{u}_0\in \Lip(\T^n)$. We choose the potential 
\begin{equation}\label{choiceV}
	    \hat V_0(x):=-\partial_rf(x,0)\, \hat u_0(x),\quad x\in\T^n.
\end{equation}
Since $f\in C^{1,1}_{\rm loc}(\T^n\times\R)$, we have $\partial_rf(x,0)\in \Lip(\T^n)$, and thus $\hat V_0\in \Lip(\T^n)$. 

For this choice of potential $\hat V_0$, the perturbed equation 
\begin{equation}\label{eq_HhatV0}
    f(x,\lambda u_\lambda)+ H(x, Du_\lambda)+\lambda \hat{V}_0(x)=\tr\big( A(x) D^{2}u_\lambda\big)+c_H \quad \text{in}~\T^n,
\end{equation}
satisfies all assumptions {\rm (H1)--(H3)} of Theorem \ref{mainresult1}. Therefore, for all sufficiently small $\lambda>0$, equation \eqref{eq_HhatV0} admits a unique Lipschitz continuous viscosity solution $u_\lambda:\T^n\to\R$, and $u_\lambda$ converges uniformly as $\lambda\to 0^+$ to a particular function $u_0$.

It remains to prove that $u_0=\hat u_0$ on $\T^n$ and the error estimate $\|u_\lambda-\hat u_0\|\sim O(\lambda)$. 

We rewrite \eqref{eq_HhatV0} as $\mathcal{F}_\lambda[u_\lambda]=0$, where, formally, 
\begin{equation*}
	\mathcal{F}_\lambda[w]:=  f(x,\lambda w) -\tr \big(A(x) D^2w\big) + H(x, Dw) +\lambda \hat{V}_0-c_H.
\end{equation*}
Define
\begin{equation*}
	u_\lambda^+:=\hat{u}_0+ M\lambda\,,
\end{equation*}
where the constant $M> 0$ will be chosen below. We claim that, for $M$ suitably large, $u_\lambda^+$ is a viscosity supersolution of \eqref{eq_HhatV0} for all sufficiently small $\lambda>0$.

Indeed, substituting $u_\lambda^+$ into $\mathcal{F}_\lambda$, we compute that
\begin{equation*}
\begin{aligned}
	\mathcal{F}_\lambda[u_\lambda^+]=~& f(x,\lambda u_\lambda^+) -\tr \big(A(x) D^2u_\lambda^+\big) + H(x, Du_\lambda^+) +\lambda \hat{V}_0-c_H\\
	=~&f(x,\lambda \hat u_0+M\lambda^2)- \lambda \partial_rf(x,0)\, \hat u_0\\
	=~&\underbrace{f(x,\lambda \hat u_0+M\lambda^2)-f(x,\lambda \hat u_0)}_{=:I_1}+\underbrace{f(x,\lambda \hat u_0)- \lambda \partial_rf(x,0)\, \hat u_0}_{=:I_2}\\
\end{aligned}	
\end{equation*}

\noindent\textbf{Estimate of $I_1$.} Since $\partial_rf(x,r)> 0$ by assumption (H3), there exists a constant $d_0=d_0(\hat u_0)>0$ such that
$
    \partial_r f(x,r)\geq d_0
$
for all $x\in\T^n$ and all $|r|\leq \|\hat u_0\|_\infty+1$. Then, 
\begin{equation}\label{I1esti}
	I_1 \geq d_0 M\lambda^2
\end{equation}
provided that $\lambda\leq 1$ and $M\lambda\leq 1$.

\noindent\textbf{Estimate of $I_2$.} Since $f(x,0)\equiv 0$ from assumption (H3), we write
\begin{equation*}
	I_2=f(x,\lambda \hat u_0(x))- f(x,0)-\lambda \partial_rf(x,0)\, \hat u_0(x)=\int_0^{\lambda \hat u_0(x)}
    \partial_r f(x,s)-\partial_r f(x,0)\diff s.
\end{equation*}
Since $f\in C^{1,1}_{\rm loc}$, there exists a constant $K_0=K_0(\hat u_0)>0$ such that \[|\partial_r f(x,s)-\partial_r f(x,0)|\leq K_0 |s|\] for all $x\in \T^n$ and all $|s|\leq \|\hat u_0\|_\infty$. Hence, for all $\lambda\leq 1$,
\begin{equation}\label{I2esti}
	|I_2| =~\left|\int_0^{\lambda \hat u_0(x)}
    \partial_r f(x,s)-\partial_r f(x,0)\diff s\right|\leq ~\frac{K_0}{2}\lambda^2\|\hat u_0\|_\infty^2\,.
\end{equation}

We now choose
\begin{equation*}
	M:=\frac{K_0}{2d_0}\|\hat u_0\|^2_\infty\,.
\end{equation*}
Combining the above two estimates \eqref{I1esti}-\eqref{I2esti} gives \begin{equation*}
	\mathcal{F}_\lambda[u_\lambda^+]\geq \lambda^2\left( d_0 M-\frac{K_0}{2}\|\hat u_0\|_\infty^2\right)=0	.
\end{equation*}
for all $\lambda\leq 1$ and $M\lambda\leq 1$. Therefore, for all sufficiently small $\lambda>0$, $u_\lambda^+:=\hat{u}_0+ M\lambda$ is a viscosity supersolution of \eqref{eq_HhatV0}. Similarly, define $u_\lambda^-:=\hat{u}_0- M\lambda$. An entirely analogous argument shows that $u_\lambda^-$ is a viscosity subsolution of \eqref{eq_HhatV0}. Using the comparison principle for equation \eqref{eq_HhatV0} gives
\[
   \hat{u}_0- M\lambda=u_\lambda^-\leq u_\lambda\leq u_\lambda^+=\hat{u}_0+ M\lambda\quad \text{on}~ \T^n
\]
for all sufficiently small $\lambda>0$. Consequently, \[
    \|u_\lambda-\hat u_0\|_{\infty} \leq M\lambda.
\]
This finally proves Theorem \ref{mainresult3}.
\end{proof}

\noindent{\bf Acknowledgments.}
This work is supported by the National Natural Science Foundation of China (No. 12301228 and No. 12571207) and the Fundamental Research Funds for the Central Universities (No. 020314380038).

\bibliographystyle{plain}%{alpha}

\end{document}